\documentclass[12pt]{amsart}

\usepackage{tikz}
\usepackage{amsmath, verbatim}
\usepackage{amssymb,amsfonts,mathrsfs}
\usepackage[colorlinks=true,linkcolor=blue,citecolor=blue]{hyperref}
\usepackage[driver=pdftex,margin=3cm,heightrounded=true,centering]{geometry}

\usepackage{mllocal} % local macro definitions, just to avoid the clutter

\numberwithin{equation}{section}

\begin{document}

\title[Heat-trace for edge Laplacians with algebraic boundary conditions]
{Heat-trace asymptotics for edge Laplacians with algebraic boundary conditions}

\author{Boris Vertman}
\address{Mathematisches Institut,
Universit\"at Bonn,
53115 Bonn,
Germany}
\email{vertman@math.uni-bonn.de}
\urladdr{www.math.uni-bonn.de/people/vertman}

\subjclass[2010]{58J52; 34B24}
\date{This document compiled on: \today.}

\begin{abstract}
We consider the Hodge Laplace operator on manifolds with incomplete edge singularities 
and an intricate elliptic boundary value theory. We single out the class of algebraic self-adjoint extensions
for the Hodge Laplacian. Our microlocal heat kernel construction for algebraic boundary conditions is 
guided by the method of signaling solutions by Mooers, though crucial 
arguments in the conical case obviously do not carry over to the setup of edges. 
We establish the heat kernel asymptotics for the algebraic
extensions of the Hodge operator on edges, and elaborate on the exotic phenomena in the 
heat trace asymptotics which appear in the case of a non-Friedrichs extension.
\end{abstract}

\maketitle
\tableofcontents

%%%%%%%%%%%%%%%%%%
\section{Introduction}
%%%%%%%%%%%%%%%%%%

Unusual new phenomena in the heat trace asymptotics in the setup of singular spaces 
have attracted a considerable interest since the explicit observations by 
Falomir, Muschietti, Pisani and Seeley in \cite{FMPS:UPZ}
as well as by Kirsten, Loya and Park in \cite{KLP:EEA}, \cite{KLP:UPR} for certain explicit regular-singular 
operators on a line segment.
\medskip

The general problem of resolvent trace asymptotics for closed extensions of general elliptic 
cone operators with sectors of minimal growth has been studied by Gil, Krainer and Mendoza in \cite{GKM:TEF, GKM1},
who gave a detailed geometric and analytic explanation of the unusual phenomena in the resolvent trace asymptotics,
with corresponding results on the heat trace expansion, if the closed extension is sectorial, and
on the structure of its zeta-function, if the closed extension is positive. \medskip

For the Hodge-Laplace operator on a manifold with an isolated conical singularity, 
new unusual phenomena have already been hinted at by Mooers in \cite{Moo:THK}, \cite{Moo:HKA}. 
However, Mooers did not elaborate in detail on the actual heat trace asymptotics, 
but rather observed certain unexpected non-polyhomogeneity properties of the heat kernel.
The present work closes this gap and derives of a full heat trace asymptotics 
for certain self-adjoint extensions of the Hodge-Laplacian in the general setup of incomplete edge singularities.
Presence of a higher dimensional edge singularity leads to various conceptually new analytical aspects which 
we address.
\medskip
 
A complete characterization of self-adjoint extensions for the Laplacian
requires a full scale elliptic theory of edge degenerate operators, see \cite{Maz:ETD} and \cite{Sch:PDO}. 
However, in this paper we consider the class of algebraic boundary conditions, 
which define self-adjoint realizations of the Hodge-Laplacians on edge manifolds,
as already employed by the author jointly with Bahuaud and Dryden \cite{BDV:THE} in context 
of non-linear parabolic equations on edge manifolds. 
\medskip

In this paper we proceed with a construction of the heat kernel for these algebraic boundary conditions, 
guided by the method of signaling solutions by Mooers \cite{Moo:HKA} in case of isolated conical singularities. 
The setup of incomplete edge singularities requires different analytic arguments at various crucial points. 
Hereby, we present (simpler) alternative arguments to \cite{Moo:HKA} at various steps in the construction.
\medskip

In this general geometric setup we recover the unusual new phenomena 
in the heat trace asymptotics, observed in \cite{FMPS:UPZ}, \cite{KLP:EEA}, \cite{KLP:UPR}
and for general elliptic cone operators with sectorial closed extensions in \cite{GKM:TEF, GKM1}. It should be noted, 
however, that in the first three papers the analysis has been 
performed independent of the earlier work by Mooers \cite{Moo:HKA}, and 
relies on a very specific exact operator structure and Bessel analysis.
\medskip

This paper is organized as follows. We first review the basic geometry of 
incomplete edge spaces in \S \ref{geometry}. We then classify certain algebraic
self-adjoint realizations for the Hodge Laplacian in \S \ref{section-algebraic}
and recall from \cite{MazVer:ATM} the asymptotic properties of the heat kernel for the Friedrichs
self-adjoint extension in \S \ref{s-asymptotics}. We study the signaling problem in \S \ref{model-signal-section}
and \S \ref{signal-section}. The solution to the signaling problem is the central ingredient 
in the construction of the heat kernel for algebraic self-adjoint 
boundary conditions, which is explained in \S \ref{heat-section}
and is basically a revision of \cite{Moo:HKA}. Finally, in \S 
\ref{trace-section} we derive the heat trace expansion 
directly from the heat kernel structure.

%%%%%%%%%%%%%%%%%%
\section{Hodge Laplacian on incomplete edge spaces}\label{geometry}
%%%%%%%%%%%%%%%%%%

We consider a compact stratified space $\overline{M}$
which is assumed to be comprised of a single top-dimensional open stratum $M^m$ and 
a single lower dimensional stratum $B^b$. By the stratification hypothesis, $B$ is a closed manifold.
Moreover, the stratification hypothesis yields an open neighbourhood $U\subset \overline{M}$ of $B$
together with a radial function $x:U \to [0,\infty)$, such that $U\cap M$ is the total space of a smooth 
fibre bundle over $B$ with an open truncated cone $\mathscr{C}(F)=(0,1)\times F$ over a compact smooth manifold $F^f$
as the trivial fibre. The restriction of the radial function $x$ to each fibre defines the radial function of that cone.
\medskip

Resolution of the stratum $B$ in $\overline{M}$ defines a compact manifold $\widetilde{M}$ with boundary $\partial M$, 
where $\partial M$ is the total space of a fibration $\phi: \partial M \to B$ with the fibre $F$. The resolution 
process is described in detail for instance in \cite{Maz:ETD}. The neighborhood $U$ lifts to a collar neighborhood 
$\U \subset \widetilde{M}$ of the boundary, which is a smooth fibration of cylinders $[0,1)\times F$ over $B$ with the radial function $x$.
Clearly $M=\widetilde{M} \backslash \partial M$.

\begin{defn}\label{d-edge}
A Riemannian manifold with an edge singularity is the open stratum $M$ together with 
a Riemannian metric $g$ such that $g=g_0+h$ over $\mathscr{U}$, where $g_0$ attains the form 
$$g_0\restriction \U\backslash \partial M=dx^2+x^2 g^F+\phi^*g^B,$$
where $g^B$ is a Riemannian metric on the closed manifold $B$, 
$g^F$ is a symmetric 2-tensor on the fibration $\partial M$ restricting to a 
fixed Riemannian metric on each fibre $F$, $|h|_{g_0}$ is smooth on $\U$ and $|h|_{g_0}=O(x^2)$ as $x\to 0$. 
\end{defn}

Similar to other discussions in the singular edge setup, see 
\cite{Alb:RIT}, \cite{BDV:THE},\cite{BahVer:YFO} and \cite{MazVer:ATM}, we consider a slightly restricted 
class of edge metrics and require $\phi: (\partial M, g^F + \phi^*g^B) \to (B, g^B)$ to be a Riemannian submersion in the following sense. 
If $p\in \partial M$, then the tangent bundle $T_p\partial M$ splits into vertical and horizontal subspaces as 
$T^V_p \partial M \oplus T^H_p \partial M$, where $T^V_p\partial M$ is the tangent space to the fibre of 
$\phi$ through $p$ and $T^H_p \partial M$ is the annihilator of the subbundle 
$T^V_p\partial M \lrcorner g^F \subset T^*\partial M$ ($\lrcorner$ meaning contraction).  
The requirement for $\phi$ to be a Riemannian submersion is the condition that the restriction of the 
tensor $g^F$ to $T^H_p \partial M$ vanishes. 

\begin{defn}\label{def-feasible}
Let $(M,g)$ be a Riemannian manifold with an edge metric. This metric $g=g_0+h$ is said to be admissible if 
 $\phi: (\partial M, g^F + \phi^*g^B) \to (B, g^B)$ is a Riemannian submersion.
\end{defn} 

In order to explain the reason behind the admissibility assumption, consider local coordinates 
$y=(y_1,...,y_{b}),b=\dim B$ on $B$ lifted to $\partial M$ and then extended inwards to $\U$. 
Let $z=(z_1,...,z_f),f=\dim F$ restrict to local coordinates on $F$ along each fibre of the boundary. 
Then $(x,y,z)$ define local coordinates on $\U \cap M$. \medskip

Consider the Hodge Laplacian $\Delta_p$ on $(M,g)$
acting on $p$-forms, and for any $y_0$ in the coordinate patch on $B$ the normal operator 
$N(x^2\Delta_p)_{y_0}$, defined as the limit of $x^2\Delta_p$ on $p$-forms with respect to the local family of dilatations 
$(x,y,z) \to (\lambda x, \lambda (y-y_0), z)$ as $\lambda \to \infty$. Under the first admissibility assumption, 
$N(x^2\Delta_p)_{y_0}$ is naturally identified with $s^2$ times the Hodge Laplacian on $p$-forms on the 
model edge $\R^+_s \times F \times \R^b$ with incomplete edge metric $g_{\textup{ie}} = ds^2 + s^2 g^F + g^B_{y_0}$,
where we identified $T_{y_0}B=\R^b$ and denote the restriction of $g^F$ to the fibres $F$ by $g^F$ again.  
\medskip

We mention that the assumption of $g^F$ to restrict to a fixed Riemannian metric on fibres $F$ is 
only used in the Friedrichs mollifier argument in Proposition \ref{S-symmetric}. The actual analysis 
of the heat kernel needs isospectrality of fibres to ensure polyhomogeneity of the heat kernel when lifted 
to the corresponding blowup space. More precisely, here we only need that the eigenvalues of the Laplacians
on fibres are constant in a fixed range $[0,1]$. \medskip
 
The remainder of the section is devoted to the explicit structure of the Hodge Laplacian, more precisely of its normal operator, 
basically drawn from \cite[\S 2.3]{MazVer:ATM}. Consider a hypersurface $S_a=\{s=a\}$ of the model edge. Its tangent bundle 
$TS_a \equiv T(\R^b \times F)$ splits into the sum of a `vertical' and `horizontal' subspaces. The first subspace is
tangent to $F$ and the latter is tangent to the Euclidean factor $\R^b$. This splitting is orthogonal, and
we obtain a bigrading  
\begin{align}\label{bigrading}
\Lambda^p(TS_a) = \bigoplus_{j + l = p} \Lambda^j(\R^b) 
\otimes \Lambda^l (TF) =: \bigoplus_{j+l = p} \Lambda^{j,l}(S). 
\end{align}
We denote by $\Omega^{j,l}(S)$ the space of sections of the corresponding summand in this bundle decomposition.  
We want to make the normal operator $N(x^2\Delta_p)_{y_0}$ explicit with 
respect to a rescaling of the form bundles, employed also in \cite{BruSee:RES}.
More precisely, for each $j,l$ with $j + l = p$, we define
\begin{align*}
& \phi_{j,l}: C^{\infty}_0(\R^+, \Omega^{j,l-1}(S) \oplus \Omega^{j, l}(S)) 
\to \Omega^p_0(\R^b \times \mathscr{C}(F)), \\
& \qquad (\eta, \mu) \longmapsto s^{l-1-f/2}\eta \wedge ds + s^{l-f/2}\mu, 
\end{align*}
where the lower index indicates the compact support of functions and differential forms, 
away from $\{x=0\}$. We denote by $\Phi_p$ the sum of these maps over all $j+l = p$. 
Let $g_{\textup{ie}} = ds^2 + s^2 g^F + g^B_{y_0}$ be a Riemannian exact edge metric 
on $\R^b \times \mathscr{C}(F)$. Then exactly as in case of 
isolated conical singularities, we obtain an isometric transformation
$$
\Phi_p\!:\!L^2([0,1], \! L^2(\bigoplus_{j+l=p} \!\Omega^{j,l-1}(S) 
\oplus \Omega^{j,l}(S), g^F + g^B_{y_0} ), ds) \!
\to L^2(\Omega^p(\R^b \times \mathscr{C}(F)), g_{\textup{ie}}).
$$
Under this transformation we find for the normal operator
\begin{align}\label{laplace}
\Phi_p^{-1}\left[s^{-2}N(x^2\Delta_p)_{y_0}\right]\Phi_p=\left(-\frac{\partial^2}{\partial s^2}+\frac{1}{s^2}(A_p-1/4)\right) + \Delta_{\R^b,y_0}, 
\end{align}
where $\Delta_{\R^b,y_0}$ is obtained from the Hodge Laplacian on $B$ on $p$-forms by freezing coefficients at $y_0\in B$, and 
$A_p$ is the nonnegative self-adjoint operator, given on $\Omega^{l-1}(F) \oplus \Omega^{l}(F)$ by 
\begin{align}\label{a}
A_p=\left(\begin{array}{cc}\Delta_{l-1,F} + (l-(f +3)/2)^2 & 2(-1)^{l}\, \delta_{l,F}\\ 2(-1)^{l}\, d_{l-1,F}& \Delta_{l,F}+ 
(l-(f +1)/2)^2\end{array}\right).
\end{align}
One motivation for this transformation is a particularly simple form of the indicial roots.
Writing the eigenvalues of $A_p$ as $\nu_j^2, \nu_j \geq 0$, with corresponding eigenform $\phi_j$, the
corresponding indicial roots of \eqref{laplace} are given by
\begin{equation}
\gamma_j^{+} = \nu_j + \frac12 \ , \quad \gamma_j^{-} = - \nu_j + \frac12 . 
\label{indroots}
\end{equation}

A similar rescaling $\Phi_p$ by powers of the defining function $x$ 
makes sense in each local coordinate chart near the singular neighborhood $\partial M$. 
Rescalings with respect to different local coordinates are equivalent up to a diffeomorphism. 
Under conjugation by $\Phi_p$, the Hodge Laplacian on $p$-forms is a perturbation of 
\eqref{laplace} with higher order terms coming from the curvature of the Riemannian submersion 
$\phi: \partial M \to B$ and the second fundamental forms of the fibres $F$. 
We denote the rescaled operator by $\Delta_p$ again, if there is no danger of confusion.

%%%%%%%%%%%%%%%%%%
\section{Algebraic boundary conditions on incomplete edges}\label{section-algebraic}
%%%%%%%%%%%%%%%%%%

In this section we consider boundary conditions at the edge which define
self-adjoint extensions of the Hodge Laplacian of an incomplete edge space $(M,g)$. 
This is basically a short exposition of the analogous discussion in \cite{BDV:THE}.
For spaces with isolated conic singularities, this was first
accomplished by Cheeger \cite{Che:SGS}.  Further studies in the conic setting 
appear in \cite{Les:OFT}; see also \cite{Moo:HKA} and \cite{KLP:EEA}. 

Let us first review the significantly simpler situation of an isolated conical singularity, i.e.
an incomplete admissible edge space $(M,g)$ with $\dim B=0$. Set $\Delta = \oplus_p \Delta_p$. 
The normal operator of $\Delta$ is again of the same structure as \eqref{laplace}
with the rescaling $\Phi=\oplus_p \Phi_p$ and the tangential operator $A=\oplus_p A_p$. Any $u \in \mathscr{D}_{\max}(\Delta)$ 
in the maximal domain\footnote{For any differential operator $P$ acting on $C^\infty_0(M,E)$ with 
values in some Hermitian vector bundle $(E,h)$, $\mathscr{D}_{\max}(P)$ is defined as the space 
of $u\in L^2(M,E;g,h)$ such that $Pu\in L^2(M,E;g,h)$, where $Pu\in L^2$ is understood in the distributional sense.
Another natural domain is the minimal domain $\mathscr{D}_{\min}(P)$ defined as the graph closure of $P$ acting on $C^\infty_0(M,E)$.} 
of $\Delta$, admits an asymptotic expansion as $x\to 0$
\begin{equation}\label{expansion-w}
\Phi^{-1}u \, \sim \, \sum_{j=1}^q \left( c_j^+[u] \psi_j^+(x,z) + c^-_j[u] \psi_j^-(x,z) \right) + 
\tilde{u}, \ \tilde{u} \in \mathscr{D}_{\min}(\Delta),
\end{equation}
where $\{\nu^2_j\}_{j=1}^q$ enumerates eigenvalues of $A$ inside the intervall $[0,1)$, in ascending order, 
$\psi_j^{\pm} \sim x^{\gamma_j^{\pm}} \phi_j$ as $x\to 0$, with the exception 
of $\psi^-_j \sim \sqrt{x}\log(x) \phi_j$ if $\nu_j=0$, where 
$\phi_j$ denotes the normalized $\nu_j^2$-eigenform of the tangential operator $A$.
The coefficients  $c_j^{\pm}[u]$ depend on $u$ only. \medskip

There is a full characterization of self-adjoint extensions of $\Delta$ by specifying algebraic 
relations between the coefficients $c_j^{\pm}$, see e.g. (\cite{Moo:THK}, Section 7). 
For this we consider the $2q$-dimensional vector space $\Lambda_q$ spanned by solutions $\{\psi_j^{\pm}\}_{j=1}^q$ and introduce a bilinear form $\w_q$ on $\Lambda_q$ by 
\begin{equation}
\begin{split}
&\w_q(\psi_j^{+},\psi_j^{-}) = - \w_q(\psi_j^{-},\psi_j^{+}) = \left \{
\begin{split} &2\nu_j, \ \nu_j >0, \\ &1, \ \ \ \, \nu_j =0, \end{split} \right. \\
&\w_q(\psi_j^{+},\psi_j^{+}) = \w_q(\psi_j^{-},\psi_j^{-}) = \w_q(\psi_i^{\pm},\psi_j^{\pm}) = 0, i \neq j.
\end{split}
\end{equation}
Subspaces of $\Lambda_q$ where the bilinear form $\omega_q$ vanishes, may be represented 
as follows. There is a \emph{Lagrangian matrix} $q \times q$ matrix $\Gamma = (\Gamma_{ij})$ with 
diagonal entries $\Gamma_{jj}=b_{jj} \psi_j^{-} + \theta_{jj} \psi_j^{+}$
and off-diagonal entries $\Gamma_{ij}= \theta_{ij} \psi_j^{+}$; the coefficients  
$b_{ij}, \theta_{ij}\in \R$ are such that either $b_{ii}=1$ or $b_{ii}=0$, where in the
latter case we require $\theta_{ii}=1$ and $\theta_{ij}=0$ for $i\neq j$. 
If $b_{jj}=0$ whenever $\nu_j=0$, we call $\Gamma$ \emph{non-logarithmic}, 
as in this case there are no so-called ``unusual'' logarithmic terms in the expansion of the heat trace (cf. \cite{KLP:EEA}). Such a matrix defines a self-adjoint domain for $\Delta$ as follows.

\begin{defn}\label{defn:DG}
The algebraic domain of the Hodge Laplacian $\Delta$ associated to a Lagrangian matrix  $\Gamma = (\Gamma_{ij})$ is defined by 
\begin{align*}
\mathscr{D}_{\Gamma}(\Delta) := \{u\in \mathscr{D}_{\max}(\Delta) \mid 
\forall \ i=1,\dots,q: \, \sum_{j=1}^q \w_q( c_j^+[u] \psi_j^+ + c^-_j[u] \psi_j^- , \Gamma_{ij})=0\}.
\end{align*}
\end{defn}

These algebraically defined boundary conditions classify all self-adjoint extensions of the Hodge
Laplacian on cones,  cf. (\cite{Moo:THK}, Section 7) and also \cite{KLP:EEA}. \medskip

Passing to the general setting of 
incomplete feasible edge spaces with $\dim B\neq 0$ presents various crucial difficulties. 
The asymptotic expansion \eqref{expansion-w} holds only in a weak sense, i.e. only when $u\in \mathscr{D}_{\max}(\Delta)$
is paired with a smooth test function over $B$. In other words, the coefficients $c^\pm_j[u]$ are of negative Sobolev regularity
in $y\in B$. Moreover, the expansion is local in the sense that there may be no global choice of elements $\psi^\pm_j$ over the edge
manifold $B$. The error term $\tilde{u}$ need not be an element of the minimal domain $\mathscr{D}_{\min}(\Delta)$
any longer, but is only a higher order term in the asymptotics, of certain Sobolev regularity. Finally, analytic arguments 
do not generally localize over the edge $B$, since $\mathscr{D}_{\max}(\Delta)$ and $\mathscr{D}_{\min}(\Delta)$
need not be closed under multiplication with smooth cutoff functions. \medskip

The fundamental tool in dealing with the listed restrictions is a mollification argument, which 
does not apply to the second order degenerate operator $\Delta$
in any obvious way unless $B$ is either zero-dimensional or Euclidean. However, the situation changes dramatically, 
when we consider first order operators. Return back to the general setting of incomplete 
admissible edge spaces and note $\Delta = D^t D$, where $D=d+\delta$ is 
the Gauss Bonnet operator of $(M,g)$. By \cite[Lemma 2.4]{MazVer:ATM}, any $u \in \mathscr{D}_{\max}(D)$ 
admits a weak asymptotic expansion as $x\to 0$
\begin{equation}\label{expansion-u}
\Phi^{-1}u \, \sim \, \sum_{j} c_j[u] \, \psi_j(x,z;y) + \tilde{u}, 
\end{equation}
where we sum over $\nu_j \neq 0$, $\psi_j \sim x^{-\nu_j+1/2} \phi_j$ as $x\to 0$, $\tilde{u}$ is a higher order remainder, such that $\tilde{u}\in \mathscr{D}_{\min}(D)$ if all coefficients $c_j[u]\equiv 0$. As before, 
the coefficients  $c_j[u]$ are of negative Sobolev regularity in $y$, in other words the 
asymptotic expansion holds only after pairing $\Phi^{-1}u$ with a smooth test function in $C^\infty(B)$. 
The Lagrange identity for $D$ acting on 
$\mathscr{D}_{\max}(D) \cap \mathscr{A}_{\textup{phg}}$ is worked out in \cite[(2.9)]{MazVer:ATM},
and similar to Definition \ref{defn:DG} we may define algebraic domains for $D$ by 
specifying linear relations $S$ between the coefficients $c_j[u]$. Each such choice $S$ gives a self-adjoint domain 
$\mathscr{D}_{S}(D)$ in case of isolated cones, and using Friedrichs mollifiers \cite{BDV:THE} proves 
its self-adjointness in case of incomplete edges. We provide the proof here for reader's convenience.

\begin{prop}\label{S-symmetric}
$\mathscr{D}_{S}(D)$ defines a self-adjoint extension of $D$.
\end{prop}

\begin{proof}
The first part of the proof is to show that $D$ is indeed symmetric on $\mathscr{D}_{S}(D)$, 
a statement that cannot be deduced as in the conical case directly, since the 
expansion \eqref{expansion-u} holds only in the weak sense.

Let $w\in \mathscr{D}_S(D)$ and $\phi$ be a cut-off function supported in a local coordinate neighborhood $(x,y)$,
such that $\phi(x,y)=\phi_1(x) \phi_2(y)$, where $\phi_1\in C^\infty_0[0,1)$ has compact support in $[0,1)$ and is identically one near $x=0$, $\phi_2$ is a smooth cutoff function supported around some $y_0\in B$. 
Under that choice we still have $u:= w \cdot \phi \in \mathscr{D}_S(D)$. 
For each coefficient $u_I$ of the form-valued $u$ and a test function $\psi \in C^\infty(B)$, supported in a local coordinate 
neighborhood, we write in local coordinates
$$
(u_I * \psi)(x,y,z)= \int_{\R^b} u_I(x,y-\wy, z) \psi(\wy) d\wy.
$$
The convolutions $u_I * \psi$ can be assembled locally back into a differential form, 
which we denote by $u * \psi$. Since the relations $S$ between the coefficients $c_j[u], u \in \mathscr{D}(D)$
are defined by linear equations, we still have $u * \psi \in \mathscr{D}_S(D)$. Moreover, due to pairing with $\psi \in C^\infty(B)$, the coefficients 
$c_j[u*\psi]$ are now smooth in $y$ and hence $u*\psi \in \mathscr{D}_S(D) \cap \mathscr{A}_{\textup{phg}}$. 
We now specify $\psi$ to be a cutoff function, compactly supported around the coordinate origin $0\in \R^b$ in local coordinates,
with $\widehat{\psi}(0)=1$, where $\widehat{\psi}(\zeta)$ denotes the Fourier transform of $\psi$. Define a sequence $\psi_\epsilon (y) := \epsilon^{-b}\psi(y/\epsilon)$, such that
$$
\widehat{\psi_\epsilon}(\zeta) = \widehat{\psi}(\epsilon \zeta) \rightarrow \widehat{\psi}(0) =1,
\ \textup{as} \ \epsilon \to 0.
$$
Set $u_\epsilon = u * \psi_\epsilon$ and $u_{\epsilon, I}= u_I * \psi_\epsilon$. Note that the
Fourier transform of each $u_I$ in $\wy$, denoted by $\widehat{u}_I$, is $L^2(dx\, d\zeta \, dz)$-integrable. 
Since $(\widehat{\psi}(\epsilon \, \cdot) - 1)$
is bounded uniformly in $\epsilon$ we obtain by dominated convergence
\begin{align}\label{molli}
\| u_{\epsilon, I} - u_{I}\|_{L^2} = \| \widehat{u}_I \left(\widehat{\psi}(\epsilon \, \cdot) - 1\right)\|_{L^2} 
 \rightarrow 0, \ \textup{as} \ \epsilon \to 0.
\end{align}
This proves $u_\epsilon \to u$ in $L^2(dxdydz)$ as $\epsilon \to 0$. Moreover, we write (convolutions understood componentwise)
\begin{align*}
D (u * \psi_\epsilon) &= (D u) * \psi_\epsilon + 
\sum_{k\in \mathscr{K}}  \int_{\R^b}  \big(a_k(y) - a_k(y-\wy)\big)D_k u(y-\wy) \psi_\epsilon (\wy) d\wy\\
&=: (D u) * \psi_\epsilon + 
\sum_{k\in \mathscr{K}}  \int_{\R^b}  \delta_{a_k}(y,\wy) D_k u(y-\wy) \psi_\epsilon (\wy) d\wy,
\end{align*}
where $a_k D_k, k \in \mathscr{K},$ is the collection of summands in $D$ with $y$-dependent coefficients,
where each $a_k\in C^\infty(\overline{M})$ and by admissibility assumptions, $D_k$ is either a first order combination of edge derivatives 
$\mathcal{V}_e = C^\infty - \textup{span} \{x\partial_x, x\partial_{y},\partial_{z}\}$,
or $D_k  \in C^\infty - \textup{span} \{\partial_y\}$.
We will show that the second sum converges to zero in $L^2(dxdydz)$ and hence by exactly the same argument as above, 
$D u_\epsilon \to D u$ in $L^2$ as $\epsilon \to 0$. This will prove that any 
locally supported $u\in \mathscr{D}_S(D)$ can indeed by approximated by
a sequence $(u_\epsilon) \subset \mathscr{D}_S(D) \cap \mathscr{A}_{\textup{phg}}$ in the graph norm.
\medskip

If $D_k$ is a first order combination of edge derivatives, by elliptic edge theory, 
\cite{Maz:ETD}, $D_ku_I \in L^2$ and hence we may apply same argument as in \eqref{molli}. The argument is
more intricate in case $D_k  \in C^\infty - \textup{span} \{\partial_y\}$. For some $D_k =\partial_{y_j}$ we compute using 
integration by parts (omit the lower index $I$)
\begin{align*}
&\int_{\R^b} \partial_{y_j} u(y-\wy) \delta_{a_k}(y,\wy) \psi_\epsilon (\wy) \, d\wy 
= \int_{\R^b}  u(y-\wy) \,  \partial_{\wy_j}\left(\delta_{a_k}(y,\wy) \psi_\epsilon (\wy)\right) d\wy =\\
&\int_{\R^b}  u(y-\wy) \,  \partial_{y_j} a_k(y-\wy) \, \psi_\epsilon (\wy) \, d\wy +
\int_{\R^b}  u(y-\wy) \delta_{a_k}(y,\wy) \partial_{\wy_j} \psi_\epsilon (\wy) \, d\wy =:I_1 + I_2.
\end{align*}

Both $u, \partial_{y_j} a_k \cdot u \in L^2$ and repetition of the  
argument in \eqref{molli} implies that $I_1$ converges to $\partial_{y_j} a_k \cdot u$ in $L^2$ as $\epsilon \to 0$. For $I_2$
we expand $a_k(y-\wy)$ in Taylor series around $y$
\begin{align*}
I_2 &= \sum_{|\A|=1}^{N-1} \frac{(-1)^{|\A|+1}}{\A !} \, \partial_y^\A a_k(y) \int_{\R^b}  u(y-\wy) \, \wy^\A \,
\partial_{\wy_j} \psi_\epsilon (\wy) \, d\wy \\
&+ \sum_{|\A|=N} \frac{(-1)^{N+1}}{\A !} \, \int_{\R^b}  \partial_y^\A a_k(y+\theta_N(\wy-y)) \, u(y-\wy) \, \wy^\A \,
\partial_{\wy_j} \psi_\epsilon (\wy) \, d\wy,
\end{align*}
for some $\theta_N\in (0,1)$\footnote{We assume that the local coordinate neighborhood in $\R^b$ around $y$
is convex.}.
The Fourier transform of $\wy^\A \partial_{\wy_j} \psi_\epsilon (\wy)$ equals 
$\epsilon^{|\A|-1} \widehat{\wy^{\A} \partial_{\wy_j}\psi}(\epsilon \zeta)$ and hence converges pointwise to zero for $|\A|\geq 2$ as $\epsilon \to 0$. 
Similarly as before in \eqref{molli}, the corresponding summands 
converge to zero in $L^2$. For $|\A|=1$ we denote with $Y_i$ 
the multiplication operator by $\wy_i$ and obtain after integrating by parts
\begin{align*}
\lim_{\epsilon \to 0}\widehat{Y_i \partial_{y_j} \psi_\epsilon}(\zeta) = 
\lim_{\epsilon \to 0}\widehat{Y_i \partial_{y_j} \psi}(\epsilon \zeta) =  \int_{\R^b} \wy_i \, \partial_{\wy_j} \psi (\wy) \, d\wy 
= - \delta_{ij} \int_{\R^b} \psi (\wy) \, d\wy = -\delta_{ij}.
\end{align*}

Similar argument as in \eqref{molli} implies now that $I_2$ converges to $(-\partial_{y_j} a_k \cdot u)$ in $L^2$ as $\epsilon \to 0$
and hence the sum $I_1 +I_2$ converges to zero. Thus any 
locally supported $u\in \mathscr{D}_S(D)$ may be approximated in the graph norm by
a sequence $(u_\epsilon) \subset \mathscr{D}_S(D) \cap \mathscr{A}_{\textup{phg}}$.
While symmetry of $\mathscr{D}_{S}(D) \cap \mathscr{A}_{\textup{phg}}$ follows by the same argument as in the conical case,
we obtain symmetry of $D$ on $\mathscr{D}_{S}(D)$ using a partition of unity $(\phi_\A)$ subordinate to coordinate charts on $\widetilde{M}$.
For any $f,g \in \mathscr{D}_{S}(D)$ we can write
\begin{align*}
\langle D f, g \rangle_{L^2} - \langle f, D g \rangle_{L^2} = 
\sum_{\A} \left( \langle D f, g \cdot \phi_\A \rangle_{L^2} - \langle f, D ( g\cdot \phi_\A) \rangle_{L^2} \right) = 0, 
\end{align*}
where the last equality follows by approximating each $g \cdot \phi_\A$  in the graph norm by a sequence in
$\mathscr{D}_{S}(D) \cap \mathscr{A}_{\textup{phg}}$, and using symmetry of $D$
on $\mathscr{D}_{S}(D) \cap \mathscr{A}_{\textup{phg}}$.

Self-adjointness on $\mathscr{D}_{S}(D)$ now follows once we establish the following relation
$$\mathscr{D}(D_{S}^*) := \{f\in \mathscr{D}_{\max}(D) \mid \forall g \in \mathscr{D}_{S}(D): 
\langle D f, g \rangle_{L^2} = \langle f, D g \rangle_{L^2} \} \subseteq \mathscr{D}_{S}(D).$$

Consider any $f \in \mathscr{D}(D_{S}^*)$ as well as a locally 
supported $g\in \mathscr{D}_{S}(D) \cap \mathscr{A}_{\textup{phg}}$, associated to an arbitrary set of smooth coefficients $c_j[g]$ 
with compact support in $\R^b$, satisfying the algebraic relations $S$. Regularity of coefficients in the asymptotic expansion of $f$ 
is not an issue any longer due to pairing with polyhomogeneous $g$ and hence, exactly as in the conical case
we deduce from $\langle D f, g \rangle_{L^2} = \langle f, D g \rangle_{L^2}$ 
that the coefficients $c_j[f]$ in the weak expansion of $f$ in that coordinate neighbourhood must satisfy the algebraic conditions 
of $\mathscr{D}_{S}(D)$. This proves $f \in \mathscr{D}_{S}(D)$.
\end{proof}

If $B$ is either zero-dimensional or Euclidean, then a similar mollification argument yields self-adjointness 
of $\mathscr{D}_{\Gamma}(\Delta)$ for the Hodge Laplacian directly, 
without the need to invoke the first order Gau\ss \, Bonnet operator. Thus below we consider only algebraic domains 
$\mathscr{D}_{\Gamma}(\Delta)$ which arise as self-adjoint realizations $D_S^*D_S$\footnote{$D_S$ denotes
the self-adjoint realization of the Gau\ss\, Bonnet operator with domain $\mathscr{D}_S(D)$.}, 
or assume that $B$ is either zero-dimensional or Euclidean. If we restrict ourselves to those domains 
$\mathscr{D}_{\Gamma}(\Delta)$ that are compatible with the decomposition 
$L^2\Omega^*=\oplus_p L^2\Omega^p$, we may define $\mathscr{D}_{\Gamma}(\Delta_p)$
in each degree $p$. We remark that the domain $\mathscr{D}(D_S^*D_S)=\mathscr{D}_{\Gamma}(\Delta)$ 
defines a non-logarithmic Lagrangian $\Gamma$. \medskip

Classification of all self-adjoint extensions of the Hodge Laplacian on incomplete edges 
is beyond the scope of the present discussion, as it rests on a detailed understanding of the 
elliptic theory of edge degenerate operators. 

%%%%%%%%%%%%%%%%%%
\section{Asymptotics of the heat kernel on edge manifolds}\label{s-asymptotics}
%%%%%%%%%%%%%%%%%%

In a joint work with Mazzeo (\cite{MazVer:ATM}, Proposition 2.5) we have identified the Friedrichs
extension of $\Delta$ as the algebraic self-adjoint extension associated to $\Gamma=\textup{diag}(\psi^+_1,..,\psi^+_q)$.
We denote the Friedrichs extension of the Hodge Laplacian by $\Delta^{\mathscr{F}}$ and explain here the 
polyhomogeneity properties of its heat kernel $\HF$ near the edge.
\medskip

We begin by recalling the definition of conormal and 
polyhomogeneous distributions on a manifold with corners, see \cite{Mel:TAP} and \cite{Mel:COC}.
For this we consider a manifold $\mathfrak{W}$ with corners and embedded boundary faces 
$\{(H_i,\rho_i)\}_{i=1}^N$ where $\{\rho_i\}$ denote the corresponding defining functions. 
For any multi-index $b= (b_1,\ldots, b_N)\in \C^N$ we write $\rho^b = \rho_1^{b_1} \ldots \rho_N^{b_N}$.  
Let $\mathcal{V}(\mathfrak{W})$ be the space of smooth vector fields on $\mathfrak{W}$ which are
tangent to all boundary faces. Then we state the following
\begin{defn}\label{phg}
A distribution $w$ on $\mathfrak{W}$ is said to be conormal
if $w\in \rho^b L^\infty(\mathfrak{W})$ for some $b\in \C^N$ and $V_1 \ldots V_l w \in \rho^b L^\infty(\mathfrak{W})$
for all $V_j \in \mathcal{V}(\mathfrak{W})$ and for every $l \geq 0$. An index set 
$E_i = \{(\gamma,p)\} \subset {\mathbb C} \times {\mathbb N}$ 
satisfies the following hypotheses:
\begin{enumerate}
\item $\textup{Re}(\gamma)$ accumulates only at $+ \infty$,
\item if $(\gamma,p) \in E_i$, then $(\gamma+j,p') \in E_i$ for all $j \in \N_0$ and $0 \leq p' \leq p$,
\item for each $\gamma$ there exists $P_{\gamma}\in \N_0$ such 
that $(\gamma,p)\in E_i$ for every $0\leq p \leq P_\gamma < \infty$.
\end{enumerate}
An index family $E = (E_1, \ldots, E_N)$ is an $N$-tuple of index sets. 
A conormal distribution $w$ is said to be polyhomogeneous on $\mathfrak{W}$ 
with index family $E$, denoted $w\in \mathscr{A}_{\textup{phg}}^E(\mathfrak{W})$, 
if $w$ is conormal and expands near each $H_i$ as
$
w \sim \sum_{(\gamma,p) \in E_i} a_{\gamma,p} \rho_i^{\gamma} (\log \rho_i)^p, \ 
\textup{when} \ \rho_i\to 0,
$
with coefficients $a_{\gamma,p}$ themselves being conormal distributions on $H_i$ and 
polyhomogeneous with index $E_j$ at any $H_i\cap H_j$. 
\end{defn}

Let $(x,y,z)$ be a local coordinate chart in the collar neighborhood $\U$
and consider a hypersurface $\mathscr{U}_{x_0} =\{x=x_0\} \cap \mathscr{U}$.
The tangent bundle of $\mathscr{U}_{x}$ splits into the sum of a `vertical' and `horizontal' subspace, 
and as in \eqref{bigrading} we find
\begin{align}
\Lambda^p(T_{(y,z)}\mathscr{U}_{x}) = \bigoplus_{l + k= p} \Lambda^l(T_yB) 
\otimes \Lambda^k (T_zF) =: \bigoplus_{l +k = p}  \Lambda^{l,k}_{(y,z)}(\mathscr{U}). 
\end{align}
Under the rescaling transformation $\Phi$ and the orthogonal splitting above, the
heat kernel $\HF$ takes over $\R^+\times \mathscr{U}^2$ values in the sections
$$ 
\bigoplus_{p=0}^m \, C^\infty \left((0,1), \bigoplus_{l + k= p} \left(  \Lambda^{l,k-1}(\mathscr{U}) \oplus 
 \Lambda^{l,k}(\mathscr{U})\right)\right). 
$$ 
Consider local coordinates $(t, (x,y,z), (\widetilde{x}, \widetilde{y}, \widetilde{z}))$, where $(x,y,z)$ 
and $(\widetilde{x}, \widetilde{y}, \widetilde{z})$ are coordinates on the two copies of $\widetilde{M}$ near the boundary. 
The heat kernel $\HF$ has non-uniform behaviour at the submanifolds ($M^2_h:=\R^+\times \widetilde{M}^2$)
\begin{align*}
\mathscr{P} &= \{ (t, (x,y,z), (\wx, \wy, \wz))\in M^2_h \mid t=0, \ x=\wx=0, \ y= \wy \}, \\
\mathscr{D} &= \{ (t, (x,y,z), (\wx, \wy, \wz))\in M^2_h \mid t=0, \ (x, y, z) =(\wx, \wy, \wz) \}.
\end{align*}
The parabolic blowup $\mathscr{M}^2_h$ of the heat space $M^2_h$ at these submanifolds
is described in detail in \cite{MazVer:ATM} and can be illustrated as in Figure \ref{heat-incomplete}. 
The boundary faces ff and td arise by blowing up $\mathscr{P}$ and $\mathscr{D}$, respectively. 
The three other boundary faces rf, lf, tf which arise from the respective lifts of
$\{x=0\}, \{\wx=0\}, \{t=0\}$.

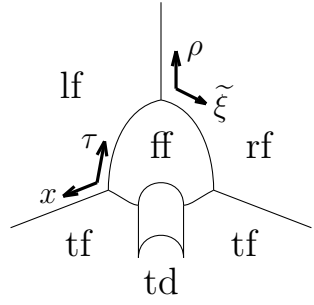
\begin{figure}[h]
\begin{center}
\begin{tikzpicture}
\draw (0,0.7) -- (0,2);
\draw (-0.7,-0.5) -- (-2,-1);
\draw (0.7,-0.5) -- (2,-1);
\draw (0,0.7) .. controls (-0.5,0.6) and (-0.7,0) .. (-0.7,-0.5);
\draw (0,0.7) .. controls (0.5,0.6) and (0.7,0) .. (0.7,-0.5);
\draw (-0.7,-0.5) .. controls (-0.5,-0.6) and (-0.4,-0.7) .. (-0.3,-0.7);
\draw (0.7,-0.5) .. controls (0.5,-0.6) and (0.4,-0.7) .. (0.3,-0.7);
\draw (-0.3,-0.7) .. controls (-0.3,-0.3) and (0.3,-0.3) .. (0.3,-0.7);
\draw (-0.3,-1.4) .. controls (-0.3,-1) and (0.3,-1) .. (0.3,-1.4);
\draw (0.3,-0.7) -- (0.3,-1.4);
\draw (-0.3,-0.7) -- (-0.3,-1.4);

\draw [very thick] (0.2,0.85) -- (0.2,1.35);
\draw [very thick] (0.2,1.35) -- (0.15,1.15);
\draw [very thick] (0.2,1.35) -- (0.25,1.15);
\draw [very thick] (0.2,0.85) -- (0.6,0.65);
\draw [very thick] (0.6,0.65) -- (0.5,0.75);
\draw [very thick] (0.6,0.65) -- (0.45,0.68);

\node at (0.8,0.65) {$\widetilde{\xi}$};
\node at (0.45,1.35) {$\rho$};

\draw [very thick] (-0.85,-0.4) -- (-1.3,-0.58);
\draw [very thick]  (-1.3,-0.58) -- (-1.1,-0.55);
\draw [very thick]  (-1.3,-0.58) -- (-1.15,-0.45);
\draw [very thick]  (-0.85,-0.4) -- (-0.75,0.15);
\draw [very thick]  (-0.75,0.15) -- (-0.72,-0.05);
\draw [very thick]  (-0.75,0.15) -- (-0.85,-0.03);

\node at (-0.95,0.15) {$\tau$};
\node at (-1.5,-0.58) {$x$};

\node at (1.3,0.1) {\large{rf}};
\node at (-1.2,0.9) {\large{lf}};
\node at (1.1, -1.2) {\large{tf}};
\node at (-1.1, -1.2) {\large{tf}};
\node at (0, -1.7) {\large{td}};
\node at (0,0.1) {\large{ff}};
\end{tikzpicture}
\end{center}
\caption{Heat-space blowup $\mathscr{M}^2_h$ for incomplete edge metrics}
\label{heat-incomplete}
\end{figure}

Instead of making the blowup procedure explicit, we choose to specify appropriate projective coordinates 
on $\mathscr{M}^2_h$. Near the top corner of ff away from tf the 
projective coordinates are given by
\begin{align}\label{top}
\rho=\sqrt{t}, \  \xi=\frac{x}{\rho}, \ \widetilde{\xi}=\frac{\widetilde{x}}{\rho}, 
\ u=\frac{y-\widetilde{y}}{\rho}, \ y, \ z, \ \widetilde{z},
\end{align}
where in these coordinates $\rho, \xi, \widetilde{\xi}$ are the defining functions of 
the faces ff, rf and lf, respectively. For the bottom corner of ff near lf, the projective coordinates are given by
\begin{align}\label{lf}
\tau=\frac{t}{x^2}, \ s=\frac{\wx}{x}, \ u=\frac{y-\widetilde{y}}{x}, \ x, \ y, \ z, \ \widetilde{z},
\end{align}
where in these coordinates $\tau, s, x$ are the defining functions of tf, lf and ff, respectively. 
For the bottom corner of ff near rf the projective coordinates are obtained by interchanging 
the roles of $x$ and $\widetilde{x}$ and are given by
\begin{align}\label{rf}
\tau=\frac{t}{\wx^2}, \ s=\frac{x}{\wx}, \ u=\frac{y-\widetilde{y}}{\wx}, \ \wx, \ \wy, \ z, \ \widetilde{z}.
\end{align}
The projective coordinates on $\mathscr{M}^2_h$ near the top of td away from tf are given by 
\begin{align}\label{td}
\eta=\sqrt{\tau}, \ S =\frac{1-s}{\eta},\ U=\frac{u}{\eta}, \  \ Z =\frac{z-\widetilde{z}}{\eta}, \  x, \ y, \ z.
\end{align}
In these coordinates tf is the face in the limit $|(S, U, Z)|\to \infty$, and ff and td are 
defined by $\widetilde{x}$ and $\eta$, respectively. The blowup heat space 
$\mathscr{M}^2_h$ is related to the original heat space $M^2_h$ via the obvious 
`blow-down map' $\beta: \mathscr{M}^2_h\to M^2_h,$
which in local coordinates is simply the coordinate change back to 
$(t, (x,y,z), (\widetilde{x}, \widetilde{y}, \widetilde{z}))$. 
\medskip

We continue under the rescaling $\Phi=\oplus_p \Phi_p$ introduced in \S \ref{geometry} and 
do not make the transformation explicit in the notation below. \medskip

We can now state the asymptotic properties of the (rescaled) $\HF$ as a 
polyhomogeneous distribution on the blowup $\mathscr{M}^2_h$, 
studied by the author jointly with Mazzeo in \cite{MazVer:ATM}.

\begin{thm}\label{thm-4-2}\textup{(\cite{MazVer:ATM}, Theorem 1.2)}
The heat kernel $\HF$ lifts under the rescaling $\Phi$ via the blowdown map $\beta$ to a 
polyhomogeneous distribution $\beta^*\HF$ on $\mathscr{M}^2_h$, with 
asymptotic expansion of leading order $(-1-\dim B)$ at the front face ff and 
$(-\dim M)$ at the diagonal face td, with index sets at the right and left boundary faces 
given by the indicial roots $\gamma \geq 1/2$ \textup{(}see \eqref{indroots}\textup{)} of the Hodge Laplacian.
\end{thm}

In fact, \cite{MazVer:ATM} went beyond the heat kernel construction, 
establishing an analogue of the even-odd 
calculus for edges with consequences for metric invariance of analytic torsion.
Below we require a rather detailed understanding of the heat kernel asymptotics
and are led to provide a short overview of the heat kernel construction in 
\cite{MazVer:ATM} for the Friedrichs extension.

\begin{defn}\label{heat-calculus}  \textup{(\cite{MazVer:ATM}, Definition 3.1)}
Let $\calE = (E_{\lf}, E_{\rf})$ be an index family for the left and right boundary faces in 
$\mathscr{M}^2_h$. Let $\Psi^{l,p,\calE}_{\eh}(M)$ 
be the space of all (rescaled by the rescaling $\Phi$) 
operators $A$ with Schwartz kernels $K_A$ which lift to polyhomogeneous functions 
$\beta^*K_A$ on $\mathscr{M}^2_h$, with index family $\{(-b-3+l +j,0): j 
\in \mathbb N_0\}$ at $\ff$, $\{ (-m + p + j, 0) : j \in \mathbb N_0 \}$ 
at $\td$, vanishing to infinite order at $\tf$ and $\calE$ for the two side faces lf and rf of 
$\mathscr{M}^2_h$. When $p = \infty$, $E_{\td} = \varnothing$.
\end{defn}

As the name suggests, kernels in the calculus $\Psi^{l,p,\calE}_{\eh}(M)$ may be composed and we state the corresponding
composition result from \cite{MazVer:ATM}.

\begin{thm}\label{composition}  \textup{(\cite{MazVer:ATM}, Theorem A.2)}
For index sets $E_{\lf}$ and $E'_{\rf}$ such that $E_{\lf}+E'_{\rf}>-1$, we have
$$\Psi^{l,p,E_{\lf}, E_{\rf}}_{e-h}(M) \circ 
\Psi^{l',\infty,E'_{\lf}, E'_{\rf}}_{e-h}(M) \subset 
\Psi^{l+l',\infty,P_{\lf}, P_{\rf}}_{e-h}(M),$$
where the index sets at the side faces of $\mathscr{M}^2_h$ amount to
\begin{equation*}
\begin{split}
P_{\lf}&=E'_{\lf}\cup (E_{\lf}+ \l') 
\cup \{(z, p + q + 1): \exists \, (z,p) \in E'_{\lf},\ \mbox{and}\  (z,q) \in (E_{\lf}+ \l') \}, \\
P_{\rf}&= E_{\rf}\cup (E'_{\rf}+ \l)
\cup \{(z, p + q + 1): \exists \, (z,p) \in E_{\rf},\ \mbox{and}\  (z,q) \in (E'_{\rf}+ \l) \}. 
\end{split}
\end{equation*} 
\end{thm}

The heat kernel construction proceeds in several steps. 
An initial heat kernel parametrix is obtained in \cite[Proposition 3.2]{MazVer:ATM}.
It is given by setting at the front face of $\mathscr{M}^2_h$ (we employ projective 
coordinates \eqref{rf})
\begin{equation}
N_\ff(H):=H^{\mathscr{C}(F)}(\tau,s,z,\widetilde{s}=1,\widetilde{z})H_{\RR^b}(\tau,u,\widetilde{u}=0;y),
\end{equation}
where $H_{\R^b}(\cdot\, ;y)$ denotes the heat kernel for $\Delta_{\R^b,y}$ obtained from the Hodge Laplacian on $B$ in local 
coordinates, by freezing coefficients at $y\in B$. $H^{\mathscr{C}(F)}$ is a the heat kernel on the model cone $(\mathscr{C}(F)=\R^+\times F, ds^2+s^2g^F)$, given by 
\begin{equation}\label{bessel-sum}
H^{C(F)}(\tau, s, z, \widetilde{s},\widetilde{z})= \sum_\nu \frac{(s \widetilde{s})^{\frac12}} {2\tau} I_{\nu}\left(\frac{s\widetilde{s}}{2\tau}\right)
e^{-\frac{s^2+\widetilde{s}^2}{4\tau}}\phi_{\nu}(z)\phi_{\nu}(\widetilde{z}),
\end{equation}
where we sum over $\nu\geq 0$ with $\nu^2\in \textup{Spec}(A)$, 
$\phi_{\nu}$ is the eigenform associated to $\nu^2\in \mbox{spec}\,(A)$.
Classical bounds for the Bessel functions show that this sum converges locally uniformly in $\calC^\infty$.
\medskip

The initial parametrix  is obtained by extending $\rho_{\ff}^{-1-b}N_\ff(H)$ smoothly off the front face.  
This defines $H^{(0)} \in \Psi^{2,0,\calE}_{\eh}(M)$, where $\calE = (E_{\lf}, E_{\rf})$ and the 
index sets $E_{\lf}= E_{\rf}$ are given by $\{(\nu + 1/2 + k, 0) \mid \nu^2 \in \mbox{spec}\,(A), k \in \N\}$.
In the next step one constructs $H^{(1)} \in \Psi^{2,0,\calE}_{\eh}(M)$ by adding correction terms near $\td$ such that
$t\calL H^{(1)} = P^{(1)} \in \Psi^{3,\infty,\calE^{(1)}}_{\eh}(M)$, where  $\calE^{(1)} = (E_{\lf}, E_{\rf}-1)$, and 
\[
\lim_{t\to 0}H^{(1)}(t, x, y, z, \widetilde{x},\widetilde{y},\widetilde{z}) = 
\delta(x-\widetilde{x})\delta(y-\widetilde{y})\delta(z-\widetilde{z}).
\]

In the next construction step one chooses a slightly finer parametrix $H^{(2)}$ with an error
which vanishes to infinite order along $\rf$ as well. This is the content of 
\cite[Proposition 3.3]{MazVer:ATM}. More precisely, there exists an element $\mathcal{J} \in \Psi^{3,0,\calE'}_{\eh}(M)$, where 
$\calE' = (E_{\lf}, E_{\rf}+1)$ and $H^{(2)}:= H^{(1)}+ \mathcal{J}$ is
such that $t\calL H^{(2)} = P^{(2)} \in \Psi^{3,\infty, E_{\lf}, \infty}_{\eh}(M)$ and $\lim_{t \to 0} H^{(2)} = \mathrm{Id}$.
The identity operator $\mathrm{Id}$ corresponds to the kernel $\delta(x-\widetilde{x})\delta(y-\widetilde{y})\delta(z-\widetilde{z})$. 
\medskip

Consider the kernels as convolution operators in time. Then, our parametrix $H^{(2)}$ solves 
$\mathcal{L}H^{(2)}= \mathrm{Id} + t^{-1}P^{(2)}$. 
The final stage in the parametrix construction is then to consider the formal Neumann series 
\[
(\mathrm{Id} + t^{-1}P^{(2)})^{-1} = \mbox{Id} + \sum_{j=1}^\infty (- t^{-1} P^{(2)})^j := \mathrm{Id} + P^{(3)},
\]
where $t^{-1} P^{(2)} \in \Psi^{1,\infty, E_{\lf},\infty}_{\eh}$ and 
$(t^{-1} P^{(2)})^j \in \Psi^{j,\infty, *,\infty}_{\eh}$ by Theorem \ref{composition}. 
By the arguments in \cite{MazVer:ATM} the exact heat kernel is then given by
\begin{equation*}
\begin{split}
\HF = H^{(2)}(\mathrm{Id} + P^{(3)})
&=H^{(2)} - H^{(2)} * t^{-1} P^{(2)} + H^{(2)} * \sum_{j=2}^\infty (- t^{-1} P^{(2)})^j \\
&=H^{(1)} + \left(\mathcal{J} + H^{(1)} * \mathcal{L}H^{(1)}\right) + \mathscr{R}, \ \mathscr{R} \in \Psi^{4,*}_{\eh}(M).
\end{split}
\end{equation*}

Recall Definition \ref{d-edge}, which requires 
$g=g_0 + h$ with $|h|_{g_0}= O(x^2)$ as $x\to 0$. We can therefore separate
$\mathcal{L}$ into the leading order term $N_\ff (\mathcal{L})$, second order term $\mathcal{L}'$, 
comprised of derivatives $\{\partial_y, \partial_y\partial_z\}$, weighted with functions smooth up to 
$\partial M$, which arise from the curvature of 
the fibration $\phi: (\partial M, g^F + \phi^*g^B) \to (B, g^B)$, and the higher order terms 
$\mathcal{L}''$, which arise from $h$ and do not lower the front face asymptotics.
Consequently
\begin{equation}\label{HJ}
\begin{split}
\HF = H^{(1)} + \left(\mathcal{J} + H^{(1)} * \mathcal{L}'H^{(1)}\right) + \mathscr{R}', \ \mathscr{R}' \in \Psi^{4,*}_{\eh}(M).
\end{split}
\end{equation}

%%%%%%%%%%%%%%%%%%
\section{Solution to the model signaling problem}\label{model-signal-section}
%%%%%%%%%%%%%%%%%%

We now proceed with the first step in the construction of the heat kernel 
$\HG$ for an algebraic self-adjoint extension $\Delta_\Gamma$ with domain $\mathscr{D}_\Gamma (\Delta)$. The fundamental idea is to add terms to the 
heat kernel $\HF$ for the Friedrichs extension of the Hodge Laplacian, which correct the
asymptotic behaviour of the kernel at $\rf$ and $\lf$ to satisfy the boundary conditions of $\mathscr{D}_\Gamma(\Delta)$. 
These additional terms are obtained from
the signaling solution, which is explained and constructed out of $\HF$ below in \S \ref{signal-section}. The present 
section provides a preliminary discussion of the signaling problem for a prototype of a model edge $l_\nu \oplus \Delta_{\R^b}$, 
where $l_\nu:=-\partial_x^2 + x^{-2}(\nu^2-1/4), \nu \in [0,1),$ is a regular-singular differential operator acting on $C^\infty_0(\R^+), \R^+:=(0,\infty)$,
and $\Delta_{\R^b}$ denotes the Hodge Laplacian on $\Omega^*(\R^b)$.
\medskip

The maximal domain $\mathscr{D}_{\max}(l_\nu)$ and the minimal domain $\mathscr{D}_{\min}(l_\nu)$ 
for $\l_\nu$ provide the maximal and minimal 
closed extensions of the regular singular operator $l_\nu$ in $L^2(\R^+)$, as introduced in \S \ref{section-algebraic}. As a special case of
\eqref{expansion-w}, we refer for instance to (\cite{KLP:EEA}, Proposition 3.1) for an explicit argument, any 
$u \in \mathscr{D}_{\max}(l_\nu)$ admits a partial asymptotic expansion 
\begin{equation*}
\begin{split}
&u \sim c^+[u] \,\phi^+_\nu(x) + c^-[u] \, \phi^-_\nu(x) +  \widetilde{u}, \ \textup{as} \ x\to 0, \\
&\ \phi^+_\nu(x) = x^{\nu+1/2}, \quad \phi^-_\nu(x) = \left\{
\begin{split}
&x^{-\nu+1/2},\ \nu\in (0,1),\\
&\sqrt{x}\log(x), \ \nu=0,
\end{split}
\right. \ \widetilde{u} \in \mathscr{D}_{\min}(l_\nu).
\end{split}
\end{equation*}
The weak expansion of solutions in $\mathscr{D}_{\max}(l_\nu \oplus \Delta_{\R^b})$ is of a parallel structure,
with coefficients $c^\pm[u]$ being distributions over $\R^b$ of negative Sobolev regularity. 
The \emph{signaling solution} $u(\cdot, t)\in \mathscr{D}_{\max}(l_\nu \oplus \Delta_{\R^b}), t\in [0,\infty)$, 
is defined for any given $h \in C^\infty(\R^+\times \R^b)$ such that $e^{-ct}h \in L^1(\R^+_t\times \R^b)$
for any $c>0$\footnote{This requirement quarantees that the Laplace transform of $h$ in the $\R^+$-variable, 
and the simultaneous Fourier transform in the $\R^b$-variable are well-defined.}
as a solution to the so-called \emph{signaling problem}
\begin{equation}
\label{sign}
\begin{split}
(\partial_t + l_\nu \oplus \Delta_{\R^b}) u(x,y,t)&=0, \quad u(x,y,0)\equiv 0, \\
c^-(u(\cdot, t))&=h(t), \ t>0.
\end{split}
\end{equation}
Note that $u(\cdot, t)$ cannot take values in the domain of a fixed self-adjoint extension of $l_\nu \oplus \Delta_{\R^b}$,  
since by uniqueness of solutions to the heat equation, $u(x,y,0)\equiv 0$ then implies $u\equiv 0$. 
The signaling solution is in fact also unique, since $\{u\in \mathscr{D}_{\max}(l_\nu \oplus \Delta_{\R^b}) \mid c^-[u]=0\}$
defines the Friedrichs self-adjoint extension of $l_\nu \oplus \Delta_{\R^b}$, which we denote by $L_\nu^{\mathscr{F}}\oplus \Delta_{\R^b}$.
\medskip

The heat kernel of the Friedrichs self-adjoint extension $L_\nu^{\mathscr{F}}$ of the regular singular 
operator $l_\nu$ and its restriction to $\{\wx=0\}$ are explicitly given by\footnote{see (\cite{Les:OFT}, Proposition 2.3.9) 
and compare to \eqref{bessel-sum}.}
\begin{equation}
\label{model}
\begin{split}
E_\nu(t, x,\wx) &= \frac{\sqrt{x\wx}}{2t} I_\nu\left(\frac{x\wx}{2t}\right) 
\exp\left(- \frac{x^2+\wx^2}{4t}\right), \\
NE_\nu(t,x) &:= \lim_{\wx\to 0}\left(\wx^{-\nu-1/2}E_\nu(x,\wx,t)\right)
=\frac{x^{\nu+1/2}\exp(-x^2/4t)}{\Gamma(\nu+1)2^{2\nu+1}t^{\nu+1}},
\end{split}
\end{equation}
where we have used the asymptotic behaviour of the modified Bessel function of first kind, 
cf. \cite{AbrSte:HOM}
$$
I_\nu(r) \sim \frac{(r/2)^\nu}{\Gamma(\nu+1)}, \ \textup{as} \ r\to 0.
$$
We define ($H_{\R^b}$ denotes the Euclidean heat kernel of the Hodge Laplacian on $\R^b$)
\begin{equation}
\begin{split}
F^N_\nu(h)(t,x,y) := c_\nu \cdot \int_0^t \int_{\R^b} NE_\nu (\wt, x) &H_{\R^b}(\wt, y-\wy) h(t-\wt, \wy) \, d\wt \, d\wy,  
\\  &c_\nu := \left\{\begin{split}
&(-1) , \ \textup{for} \ \nu=0, \\ & 2\nu, \ \textup{for} \ \nu \in (0,1).
\end{split}\right.
\end{split}
\end{equation}
for any $h \in C^\infty(\R^+\times \R^b)$ such that $e^{-ct}h \in L^1(\R^+_t\times \R^b)$
for any $c>0$. The asymptotic behavior of $F^N_\nu(h)$ as $x\to 0$ is studied by means of the Laplace transform $\mathscr{L}$ in the time variable $t\in \R^+$ and the Fourier transform $\mathscr{F}$ in the Euclidean variable $y\in \R^b$.
For any $g \in C^\infty(\R^+\times \R^b)$ such that $e^{-ct}g \in L^1(\R^+_t\times \R^b)$
for any $c>0$, both transforms are defined as follows
\begin{align}\label{laplace-transform}
&(\mathscr{L}g)(\zeta, y) = \int_{\R^+} g(t,y)\exp(-\zeta t) dt, \textup{Re}(\zeta) >0, \\
&(\mathscr{F}g)(t,\w) = \int_{\R^b} g(t,y) e^{-i\w \cdot y} dy.
\end{align}
Hence we assume henceforth that $h \in C^\infty(\R^+\times \R^b)$ such that 
$e^{-ct}h \in L^1(\R^+_t\times \R^b)$ for any $c>0$.
Finally, the inverse Laplace transform is given for any any $\delta >0$ and analytic $L(\zeta)$, 
integrable over $\textup{Re}(\zeta) = \delta$, by
\begin{align}\label{laplace-inverse}
(\mathscr{L}^{-1}L)(t) = \frac{1}{2\pi i}\int_{\delta + i\, \R} e^{t\zeta} L(\zeta) \, d\zeta.
\end{align} 

\begin{prop}\label{model-F} $F^N_\nu(h)$ is indeed the signaling solution to \eqref{sign}. 
In particular\footnote{We fix the main branch of the logarithm on $\C\backslash \R^-, \R^-=(-\infty,0]$
throughout this paper.}
\begin{equation}
\begin{split}
c^-(F^N_\nu(h)) = h, \ c^+(F^N_\nu(h)) &= G^N_\nu * h(t) := \int_0^t \int_{\R^b} G^N_\nu (t-\wt, y-\wy) h(\wt, \wy) \, d\wt\, d\wy, \\
 \textup{where} \ \left(\mathscr{L} \circ \mathscr{F} G^N_\nu\right)(\zeta, \w) 
&\equiv \widetilde{G}^N_\nu(\zeta + |\w |^2) \\ &:= \left\{
\begin{split}
& \log \sqrt{\zeta + |\w |^2} + \gamma - \log 2, \ \nu=0, \\
& \frac{\Gamma(-\nu)}{\Gamma(\nu)}\, 2^{-2\nu}\, (\zeta + |\w |^2)^\nu , \ \nu \in (0,1).
\end{split}
\right.
\end{split}
\end{equation}
\end{prop}

\begin{proof}
We compute for $\textup{Re}(\zeta)>0$
\begin{align*}
(\mathscr{L}\circ \mathscr{F} F^N_\nu(h))(x,\w, \zeta) &=  c_\nu \cdot \mathscr{L} (NE_\nu \cdot \mathscr{F}H_{\R^b})(x,\w ,\zeta) 
\cdot  (\mathscr{L}\circ \mathscr{F}h)(\zeta,\w) \\
&=\int_0^\infty NE_\nu (t,x) e^{-t(\zeta + |\w |^2)} dt \cdot  (\mathscr{L}\circ \mathscr{F}h)(\zeta,\w)
\\ &= c_\nu \frac{ \sqrt{x} \, (\zeta + |\w |^2)^{\nu/2}} {2^\nu \Gamma (\nu+1)} \, K_\nu(x\sqrt{\zeta + |\w |^2})  
\cdot (\mathscr{L}\circ \mathscr{F}h)(\zeta,\w),
\end{align*}
where $K_0$ is the modified Bessel function of second kind, and in the definition of $\sqrt{\zeta}$
we fix the branch of logarithm in $\C\backslash \R^-$. By assumption on $h$, 
$(\mathscr{L}\circ \mathscr{F}h)(\zeta,\w)$ is well-defined for $\textup{Re}(\zeta)>0$. 
The Bessel function $K_\nu(z)$ admits an asymptotic expansion, see \cite{AbrSte:HOM} 
\begin{equation*}
K_\nu(z) = \left\{
\begin{split}
& (\log 2- \gamma ) - \log(z), \ \textup{for} \ \nu=0, \\
& 2^{ \nu-1}\Gamma(\nu) \, z^{-\nu}+ 2^{-\nu-1} \Gamma(-\nu) \,z^{\nu}, \ \nu>0,
\end{split}
\right\} + \widetilde{K}(z), \widetilde{K}(z)=O(z) \ \textup{as} \ z\to 0.
\end{equation*}
where $\gamma\in \R$ is the Euler constant. Consequently 
\begin{equation*}
\begin{split}
&c^+(\mathscr{L}\circ \mathscr{F} F^N_\nu(h)(\cdot,\zeta, \w)) = c_\nu \cdot (\mathscr{L}\circ \mathscr{F}h)(\zeta,\w) \cdot \left\{
\begin{split}
&(\log 2 -\log \sqrt{\zeta + |\w |^2} - \gamma),\ \nu=0, \\
&\frac{\Gamma(-\nu)(\zeta + |\w |^2)^\nu}{\Gamma(\nu+1)\, 2^{2\nu+1}}, \ \nu\in (0,1),
\end{split}\right. \\
&c^-(\mathscr{L}\circ \mathscr{F}F^N_\nu(h)(\cdot,\zeta, \w)) =  (\mathscr{L}\circ \mathscr{F}h)(\zeta,\w).
\end{split}
\end{equation*}
Taking the inverse Laplace and Fourier transform, we obtain 
\begin{align*}
F^N_\nu(h)(x,t) &=\phi^+_\nu(x) (\mathscr{L}\circ \mathscr{F})^{-1}(c^+(\mathscr{L}\circ \mathscr{F} F^N_\nu(h))) \\ &+ 
\phi^-_\nu(x)(\mathscr{L}\circ \mathscr{F})^{-1}(c^-(\mathscr{L}\circ \mathscr{F} F^N_\nu(h))) \\ &+ 
 \frac{ \sqrt{x}} {2^\nu \Gamma (\nu+1)}(\mathscr{L}\circ \mathscr{F})^{-1}((\mathscr{L}\circ \mathscr{F} h)\, \widetilde{K}(x\sqrt{\zeta+|\w |^2})
 (\zeta + |\w |^2)^{\nu/2}),
\end{align*}
where each $\phi^\pm_\nu$ coefficient exists and 
the third summand is $O(x^{3/2})$, as $x\to 0$. Consequently, indeed 
$$\mathscr{L}\circ \mathscr{F}(c_{\pm}(F^N_\nu(h))) =c_{\pm} (\mathscr{L}\circ \mathscr{F} F^N_\nu(h)).$$
This yields the stated explicit expression for the Laplace-Fourier transform of $G^N_\nu$
\begin{equation}\label{LG}
(\mathscr{L} \circ \mathscr{F} G^N_\nu)(\zeta, \w) =  \left\{\begin{split}
& \log \sqrt{\zeta + |\w |^2} + \gamma - \log 2, \ \nu=0, \\
& \frac{\Gamma(-\nu)}{\Gamma(\nu)}\, 2^{-2\nu}\, (\zeta + |\w |^2)^\nu , \ \nu \in (0,1).
\end{split} \right.
\end{equation}
\end{proof}

%%%%%%%%%%%%%%%%%%
\section{Solution to the signaling problem}\label{signal-section}
%%%%%%%%%%%%%%%%%%

We expand the lifted heat kernel $\beta^*\HF$ (as before rescaled under $\Phi$ from \S \ref{geometry}) 
asymptotically at the left boundary face,
using projective coordinates \eqref{lf}, where $s=\wx / x$ is the defining function of $\lf$ 
and $x$ the defining function of the front face. We obtain by Theorem \ref{thm-4-2}
\begin{align*}
\beta^*\HF \sim \sum_{\nu\geq 0} \sum_{k \in \N_0} G^k_{\nu}\, s^{\nu+1/2+k}, \ \textup{as} \ s\to 0,
\end{align*}
where the sum runs over $\nu\geq 0$ with $\nu^2$ being the eigenvalues of the tangential operator 
$A=\oplus_p A_p$, cf. \eqref{a}, and natural numbers $k\geq 0$. The coefficient $G^k_{\nu}$
is a polyhomogeneous function on the left boundary face, of leading order $(-1-b)$
at the front face and vanishing to infinite order at $\tf$. In the projective coordinates \eqref{lf}, we find 
$$
\beta^*(x\partial_x) = x\partial_x - s\partial_s, \ \beta^*(x\partial_y)= \partial_u, \ \beta^*(\partial_z)= \partial_z,
\ \beta^*(x^2\partial_t) = \partial_\tau.
$$
Hence the action of $\beta^*(x^2(\partial_t + \Delta))$ keeps the order of $s^\A$-terms invariant. Consequently, since
$\HF$ solves the heat equation, so does each coefficient $G^k_{\nu}\, s^{\nu+1/2+k}$ in the heat kernel expansion 
at $\lf$. We define
\begin{align}
G^0_{\nu}\, s^{\nu+1/2} = G^0_{\nu} \, \wx^{\nu+1/2}x^{-\nu-1/2}=:H_\nu \, \wx^{\nu+1/2},
\end{align}
where $H_\nu$ solves heat equation, since so does $G^0_{\nu}\, s^{\nu+1/2}$. The kernel $H_\nu$ 
lifts to a polyhomogeneous function on the left face $\lf$ of leading order $(-3/2-b-\nu)$
at the front face and vanishing to infinite order at $\tf$. \medskip

The left face $\lf$ is itself a parabolic blowup of $\R^+_t \times \R^+_x \times \partial M_{(y,z)}
\times \partial M_{(\wy, \wz)}$ at $\{(t,x,y,z,\wy,\wz) \in (\R^+)^2\times (\partial M)^2 \mid t=x=0, y=\wy\}$.
The left face $\lf$ is a fibration over $B$ of manifolds with corners, and at each $y\in B$ its fibre may be 
illustrated as 

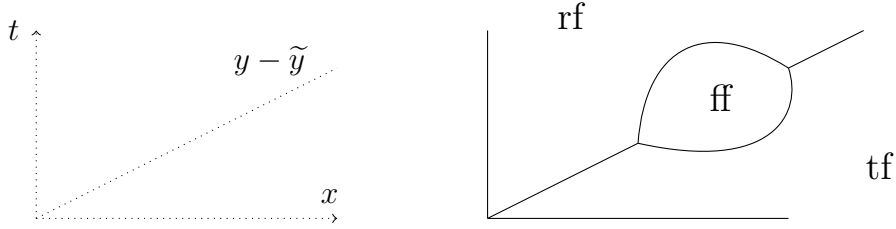
\begin{figure}[h]
\begin{center}
\begin{tikzpicture}
\draw (0,0) -- (2,1);
\draw (4,2) -- (5,2.5);
\draw (0,0) -- (4,0);
\draw (0,0) -- (0, 2.5);
\draw (2,1) .. controls (2.1,2.3) and (2.9,2.7) .. (4,2);
\draw (2,1) .. controls (3.8,0.6) and (4.2,1.4) .. (4,2);
\node at (3.1,1.6) {\large{ff}};
\node at (1.1,2.7) {\large{rf}};
\node at (5.2,0.7) {\large{tf}};

\draw[dotted] (-6,0) -- (-2,2);
\draw[dotted, ->] (-6,0) -- (-2,0);
\draw[dotted, ->] (-6,0) -- (-6, 2.5);

\node at (-2.1,0.3) {$x$};
\node at (-6.3,2.5) {$t$};
\node at (-2.9,2.1) {$y-\wy$};

\end{tikzpicture}
\end{center}
\label{figure-edge}
\caption{Fibre of $\lf$ at $y\in B$ as a blowup of $(\R^+)^2\times F\times \partial M$.}
\end{figure} 

with projective coordinates near $\ff$ away from $\rf$ given by 
\begin{align}\label{coord-left-1}
\tau = \frac{\wt}{x^2}, \ u=\frac{y-\wy}{x}, \ x, \ y, \ z, \ \wz,
\end{align}
where $x$ is the defining function of the front face $\ff$ and $\tau$ is the 
defining function of the temporal face $\tf$. Projective coordinates near 
$\ff$ away from $\tf$ are given by
\begin{align}\label{coord-left-2}
\rho=\sqrt{\wt}, \ \xi=\frac{x}{\rho}, \ u=\frac{y-\wy}{\rho}, \ y, \ z, \ \wz,
\end{align}
where $\rho$ is the defining function of $\ff$ and $\xi$ is the 
defining function of the right face $\rf$. In local coordinates, the blowdown map
$\beta:\lf \to (\R^+)^2\times (\partial M)^2$ is simply the change back to standard coordinates
$(t,x,y,z,\wy,\wz)$. \medskip

Globally, $H_\nu$ is thus defined by the regularized limit in the sense that divergent asymptotic terms are neglected
$$
\beta^*H_\nu = \underset{\rho_{\lf}\to 0}{\textup{reg-lim}} (\beta^*\HF \cdot \rho_{\lf}^{-\nu-1/2}) \cdot \rho_\ff ^{-\nu-1/2}.
$$
The relation between the kernels \eqref{model} in the model situation and the 
kernels $\HF, H_\nu$ in the setup of an admissible edge manifold is then as follows. 
According to the heat kernel construction in \S \ref{s-asymptotics}, 
the expansion of the lifts $\beta^*\HF$ and $\beta^*H_\nu$ at $\ff$ in 
projective coordinates 
\eqref{top} at the top corner of $\mathscr{M}^2_h$ is given by
\begin{equation}\label{normal}
\begin{split}
&\beta^*\HF = \rho^{-1-b} 
\sum_{\nu\geq 0} E_\nu(1,\xi,\widetilde{\xi}) P_\nu(z, \wz;\wy) H_{\R^b}(1, u;y) +\beta^*K_1 + \beta^*K_2 \\
&\beta^*H_\nu = ( \rho^{-3/2-b-\nu} NE_\nu(1,\xi) H_{\R^b}(1,u;y) 
+ \beta^*\kappa_1 + \beta^*\kappa_2 ) P_\nu(z, \wz;\wy), 
\end{split}
\end{equation}
as $\rho \to 0$, where $K_{1,2}$ and $\kappa_{1,2}$ are the higher order terms, 
the sum runs over $\nu \geq 0$ with $\nu^2\in \textup{Spec}(A)$ and 
$P_\nu(z, \wz;\wy)$ is the Schwartz kernel of the fibrewise projection onto the
corresponding eigenspaces at $\wy\in B$. Since the asymptotic terms of $\beta^*\HF$
at $\rf$ in the neighborhood of the front face arise from convolution with $H^{(1)}$, 
the restriction of $\beta^*\HF$ to $\rf$, and by symmetry also to $\lf$, has the projection $P_\nu$ as a factor
in each summand of its front face expansion. Thus, each summand in the front face expansion of 
$\beta^*H_\nu$ indeed has the projection $P_\nu$ as a factor. Moreover, as $(\xi, \widetilde{\xi}, \rho) \to 0$, the higher order terms satisfy
\begin{equation}\label{higher}
\begin{split}
&\beta^*K_1 =  \sum_{\nu \geq 0} \sum_{j=0}^\infty O(\rho^{-b} (\xi \, \widetilde{\xi})^{\, \nu+1/2+j}), 
\qquad \beta^*\kappa_1 = O(\rho^{-1/2- b-\nu}\, \xi^{\, \nu+1/2}), \\
&\beta^*K_2 =\sum_{\nu \geq 0} \sum_{j=0}^\infty O(\rho^{-b+1} (\xi \, \widetilde{\xi})^{\, \nu+1/2+j}), \
\ \beta^*\kappa_2=O(\rho^{1/2 -b-\nu}\, \xi^{\, \nu+1/2}).
\end{split}
\end{equation}

We define for any $\nu \geq 0$ with $\nu^2\in \textup{Spec}(A) \cap [0,1)$, and any
$h\in C^\infty(\R^+\times \partial M)$ with $e^{-ct}h\in L^1(\R^+_t\times \partial M)$ for any $c>0$
\begin{align*}
F_\nu(h)(t,x,y,z) := c_\nu \int_0^t \int_{\partial M}H_\nu(\wt, p, \widetilde{p}) \, h(t-\wt, \widetilde{p}) \, 
d\wt \, d\textup{vol}_{\partial M}(\widetilde{p})  =: c_\nu H_\nu * h.
\end{align*}
Since $H_\nu$ solves the heat equation, so does $F_\nu(h)$. 
\medskip

The fundamental component in the heat kernel construction 
of Mooers in \cite{Moo:HKA} is a solution $u(t,\cdot) \in \mathscr{D}_{\max}(\Delta)$ 
to the \emph{signaling problem}
\begin{equation}\label{signal}
\begin{split}
(\partial_t + \Delta) u&=0, \quad u(0,\cdot)\equiv 0, \\
c^-(u(t, \cdot))&=P_\nu h(t),
\end{split}
\end{equation}
where $P_\nu$ is the fibrewise projection onto the $\nu^2$-eigenspace
of the tangential operator $A$, cf. \eqref{a}. Note that for $P_\nu h\neq 0$, the solution $u(t,\cdot)$ cannot 
lie in any fixed self-adjoint domain of $\Delta$ for all $t>0$, since by uniqueness of solutions to 
the heat equation $u(0,\cdot)\equiv 0$ then implies $u\equiv 0$. 
The signaling solution is in fact also unique, since if $c^-(u(t, \cdot))=0$, 
then $u(t,\cdot) \in \mathscr{D}(\Delta^{\mathscr{F}})$ for $t>0$, and hence 
$u(0,\cdot)\equiv 0$ then implies $u\equiv 0$.

\begin{thm}\label{FG} $F_\nu(h), \nu^2\in \textup{Spec}(A) \cap [0,1)$, 
is a signaling solution to \eqref{signal}. More precisely,
in the notation analogous to \eqref{expansion-w}, we find
\begin{equation}
\begin{split}
&F_\nu(h)(x) \sim c^+(F_\nu(h))\psi^+_\nu(x) + c^-(F_\nu(h))\psi^-_\nu(x) + O(x^{3/2}), \ x\to 0,
\\ &c^+(F_\nu(h)) = (G^N_\nu + G'_{\nu})* P_\nu h + G''_\nu *  \partial_y P_\nu h + \theta_\nu \cdot P_\nu h, 
\\ &c^-(F_\nu(h)) =P_\nu h,
\end{split}
\end{equation}
where $G^N_\nu$ was introduced in Proposition \ref{model-F} and 
$G'_\nu, G''_\nu$ lift to polyhomogeneous functions on the parabolic blowup of
$\R^+\times B^2$ around $Y:=\{(t, y, \wy) \in \R^+\times B^2 \mid  t=0, y=\wy\}$
of leading order $(-2\nu-b)$ at the front face. $\theta_\nu \in \R$ is a constant and zero unless $\nu=1/2$.
The projective coordinates on $[\R^+\times B^2, Y]$ near its front face are 
\begin{align}\label{left-coord}
\rho=\sqrt{t}, \ u=\frac{y-\wy}{\sqrt{t}}, \ y.
\end{align}
\end{thm}

\begin{proof} Choose a cutoff function $\phi \in C^\infty_0(\lf)$ with compact support, 
such that $\phi\equiv 1$ in an open neighborhood of $\ff$, and moreover $y$
and $\wy$ lie in the same coordinate chart of $B$ if $\beta^{-1}(t,x,y,z,\wy,\wz)\in \textup{supp}\, \phi$.
The kernel $(1-\phi)\beta^*H_\nu$ is of leading order $(\nu+1/2)$ at $\rf$, vanishing identically 
in an open neighborhood of $\ff$, and hence contributes to $c^+(F_\nu(h))$ by
$$
\int_0^t \int_B G(\wt, y, \wy) P_\nu h(t-\wt, \wy, z) \, d\wt\, \textup{dvol}_B(\wy),
$$ 
where $G$ lifts to a polyhomogeneous function on the parabolic blowup $[\R^+\times B^2, Y]$,
vanishing to infinite order at the front face of the blowup at $Y$. It remains to study
the contribution by $\phi \beta^*H_\nu$, where by choice of the 
cutoff function $\phi$ we may assume that $h$ is compactly supported in a coordinate chart around $y\in B$,
so that the contribution to $F_\nu(h)$ is given by
$$
F:=c_\nu \int_0^t \int_{\R^b\times F} (\beta^{-1})^*\phi \cdot H_\nu (\wt, x, y, \wy, z, \wz) 
h(t-\wt, \wy, \wz) \, d\wt \, \textup{dvol}_{\partial M}(\wy, \wz).
$$
Rewriting the integrand in projective coordinates around $\ff$, we may separate
out the $\ff$ leading order term, and integrating first in $\wz$ along the fibres
and expanding then $\textup{dvol}_B(\wy)$ around $y$, we obtain

\begin{align*}
F &= F^N_\nu(h) +  c_\nu \int_0^t \int_{\R^b} (\beta^{-1})^*\phi \cdot (\kappa_0 + \kappa_1 + \kappa_2)
(\wt,x,y,\wy)P_\nu h(t-\wt, \wy,z) \, d\wt \, \textup{dvol}_B(\wy) \\
&=: F^N_\nu(h) + F_0 + F_1 + F_2,
\end{align*}
where $\beta^*\kappa_{0,1} = O(\rho_\ff^{-\nu-b-1/2}\rho_\rf^{\nu+1/2}\rho_\tf^{\infty})$ and
$\beta^*\kappa_{2} = O(\rho_\ff^{-\nu-b+1/2}\rho_\rf^{\nu+1/2}\rho_\tf^{\infty})$. Below
we omit $(\beta^{-1})^*\phi$ from notation, and simply assume that the kernels $\kappa_{0,1,2}$
are supported in an open neighborhood of $\ff$.
\medskip

The contribution to $c^\pm(F_\nu(h))$ coming from the first integral $F^N_\nu(h)$
is studied in Proposition \ref{model-F}. It remains to discuss the latter three integrals $F_{0,1,2}$.
Since we expand as $x\to 0$, for fixed $t>0$, we may assume $x^2<t$ and separate for each $j=0,1,2$ 
\begin{align*}
F_j = \int_0^{x^2} + \int_{x^2}^t =: F'_j + F''_j.
\end{align*}
We describe the integral expressions $F'_j$ in the projective coordinates \eqref{coord-left-1}.
Then, writing $\textup{dvol}_B(\wy) = v(\wy) d\wy$ in local coordinates, we find
\begin{equation*}
\begin{split}
F'_j = \int_0^1 \int_{\R^b} G'_j(x,\tau,u,y,z) &P_\nu h (t-x^2\tau, y-xu, z) v(y-xu) \, d\tau \, du
\\ &\times \left\{\begin{split}&x^{-\nu+3/2}, \ j=0,1, \\ &x^{-\nu+5/2}, \ j=2,\end{split}\right.
\end{split}
\end{equation*}
where $G'_j$ is bounded, polyhomogeneous and vanishing to infinite order as $\tau \to 0, |u|\to \infty$.
Expanding $P_\nu h$ and $v$ in Taylor series around $(t,y,z)$, as well as expanding $G'_j$ as $x\to 0$, 
we find no contribution to terms $x^{\pm \nu +1/2}$
and $\sqrt{x}\log(x)$ in the asymptotic expansion as $x\to 0$, unless $\nu=1/2$. 
\medskip

In case $\nu=1/2$, we have $(-\nu+3/2)=\nu+1/2$ 
and hence $F'_j$ contributes $\theta'_j \cdot P_\nu h$
to the coefficient $c^+(F_\nu(h))$ if $j=0,1$, with no contribution to $c^-(F_\nu(h))$.
In case $\nu\neq 1/2$, neither of $F'_j$ contributes to the coefficients $c^\pm(F_\nu(h))$
and we set $\theta'_j=0$.
\medskip

For the analysis of $F''_j, j=0,1,2,$ consider the projective coordinates \eqref{coord-left-2}.
\medskip

\textbf{Contribution from $F''_0$.} Expanding $v(y-\rho u)$ in Taylor series around $y$, 
we find in projective coordinates \eqref{coord-left-2}
\begin{align*}
F''_0 \sim \sum_{k=1}^\infty \frac{(-1)^k}{k!} \int\limits_x^{\sqrt{t}} \int\limits_{\R^b} NE_\nu(1,\xi) & H_{\R^b} (1,u;y)
\rho^{-\nu-1/2} (\rho u)^k v^{(k)} (y)  \\ &\times P_\nu h (t-\rho^2, y-\rho u, z) \, d\rho \, du.
\end{align*}
Note that $uH_{\R^b}(1,u;y) = -\frac{1}{2} \partial_u H_{\R^b}(1,u;y)$ and hence integrating by parts in $u$, we find
\begin{align*}
F''_0 \sim \sum_{k=1}^\infty \frac{(-1)^k}{2k!} \int\limits_x^{\sqrt{t}} 
\int\limits_{\R^b} NE_\nu(1,\xi) &H_{\R^b} (1,u;y) \rho^{-\nu+1/2+k}  v^{(k)} (y)
u^{k-1} \\ &\times (\partial_y P_\nu h) (t-\rho^2, y-\rho u, z) \, d\rho \, du
\end{align*}
In view of the explicit structure $NE_\nu(1,\xi) = C\,\xi^{\nu+1/2} e^{-\xi^2/4}=C\, x^{\nu+1/2}\rho^{-\nu-1/2}
e^{-(x/2\rho)^2}$ for an explicit constant $C=\Gamma(\nu+1)^{-1}2^{-2\nu-1}$, we obtain 
\begin{align*}
F''_0 &= x^{\nu+1/2} \int\limits_x^{\sqrt{t}} \int\limits_{\R^b} \rho^{-2\nu + 1}  G''_0(\rho,u,y,z) 
(\partial_y P_\nu h) (t-\rho^2, y-\rho u, z) e^{-(x/2\rho)^2}\, d\rho \, du \\
&= \sum_{k=0}^\infty \frac{(-1)^k}{4^k k!}x^{\nu+1/2+2k} \int\limits_x^{\sqrt{t}} 
\int\limits_{\R^b} \rho^{-2\nu + 1-2k}  G''_0(\rho,u,y,z) (\partial_y P_\nu h) (t-\rho^2, y-\rho u, z) 
\, d\rho \, du \\ &=:  \sum_{k=0}^\infty F''_{0k},
\end{align*}
We may estimate
$$
\sum_{k=1}^\infty |F''_{0k}| \leq \textup{const} \sum_{k=1}^\infty \frac{x^{\nu+1/2+2k}}{4^k k!}  
\int\limits_x^{\sqrt{t}} \rho^{-2\nu + 1 - 2k} d\rho \leq \textup{const'} \, x^{-\nu+5/2},
$$
and hence this sum does not contribute to $c^\pm(F''_0)$-coefficients. 
It remains to study
\begin{align*}
F''_{00} = x^{\nu+1/2} \int\limits_x^{\sqrt{t}} 
\int\limits_{\R^b} \rho^{-2\nu + 1}  G''_0(\rho,u,y,z) (\partial_y P_\nu h) (t-\rho^2, y-\rho u, z) \, d\rho \, du
= \int\limits_0^{\sqrt{t}} \int\limits_{\R^b} - \int\limits_0^{x} \int\limits_{\R^b},
\end{align*}
using the fact that $\rho^{-2\nu+1}$ is integrable at zero for $\nu\in [0,1)$. 
Here, $G''_0$ is bounded in its components and polyhomogeneous in $\rho$.
Obviously, the latter summand does not contribute to the coefficients of $x^{\pm \nu+1/2}$
and $\sqrt{x}\log(x)$ and hence we finally obtain (we abuse the notation by incorporating the 
$\rho-$factors into the kernel $G''_0$)
\begin{align*}
c^-(F''_0)=0, \quad c^+(F''_0)= G''_0 *_t \partial_y P_\nu h = \int_0^t \int_{B} G''_0(\wt, y,\wy)
\partial_y P_\nu h (t-\wt, \wy, z) \, d\wt \, \textup{dvol}_B(\wy),
\end{align*}
where $G''_0$ lifts to a polyhomogeneous function on the parabolic blowup of
$\R^+\times B^2$ around $Y:=\{(t, y, \wy) \in \R^+\times B^2 \mid  t=0, y=\wy\}$
of leading order $(-2\nu-b)$ at the front face. \medskip

\textbf{Contribution from $F''_1$.} Here a more detailed information 
on the structure of $\kappa_1$ is necessary. Recall \eqref{HJ}, which 
asserts that the second order term in the front face expansion of $\HF$ is given by
$\mathcal{J} + H^{(1)} * \mathcal{L}' H^{(1)}$, where $\mathcal{J} \in \Psi^{3,0,\calE'}_{\eh}(M)$ with
$\calE' = (E_{\lf}, E_{\rf}+1)$, and $H^{(1)} \in \Psi^{2,0,\calE}_{\eh}(M)$ with
$\calE = (E_{\lf}, E_{\rf})$. Consequently we may write
$$
\kappa_1 = \kappa_J  +  H^{(1)} * \partial_y \kappa_L,
$$ 
where $\kappa_J$ lifts to a polyhomogeneous function of $\lf$ 
of leading order $(-1/2-\nu-b)$ at the front face and of order $(\nu+3/2)$ at the right boundary face.
Similarly, $\kappa_L$ lifts to a polyhomogeneous function of $\lf$ 
of leading order $(-3/2-\nu-b)$ at the front face and of order $(\nu+1/2)$ at the right boundary face.
Consequently we may write
\begin{align*}
F''_1
&= c_\nu \int_{x}^{\sqrt{t}} \int_{\R^b} \kappa_J(\rho, \xi, u, y) P_\nu h(t-\rho^2, y-\rho u, z)\,  \rho^{1+b} d\rho \, du \\
&+ c_\nu \int_{x}^{\sqrt{t}} \int_{\R^b} H^{(1)} * (\rho^{-1}\partial_u + \partial_y) \kappa_L (t-\rho^2, x,\rho, u, y) \, 
P_\nu h (t-\rho^2, y-\rho u, z) \rho^{1+b} d\rho\, du,
\end{align*}
where we have neglected the factor $v(\wy)$ in the volume form $\textup{dvol}_B(\wy)=v(\wy) d\wy$.
Integrating by parts in the last integral we arrive by the composition law in Theorem \ref{composition} at the following expression
\begin{align*}
F''_1 
&= x^{\nu+3/2} \int_{x}^{\sqrt{t}} \int_{\R^b} \rho^{-2\nu -1} G_J(\rho, \xi, u, y) P_\nu h(t-\rho^2, y-\rho u, z)\,  d\rho \, du \\
&+ x^{\nu+1/2} \int_{x}^{\sqrt{t}} \int_{\R^b} \rho^{-2\nu +1} G_{L,a}(\rho, \xi, u, y) P_\nu h(t-\rho^2, y-\rho u, z)\,  d\rho \, du \\
&+ x^{\nu+1/2} \int_{x}^{\sqrt{t}} \int_{\R^b} \rho^{-2\nu +1} G_{L,b}(\rho, \xi, u, y) (\partial_y P_\nu h)(t-\rho^2, y-\rho u, z)\,  d\rho \, du,
\end{align*}
where the kernels $G_J, G_{L,a}, G_{L,b}$ are bounded and polyhomogeneous in $\rho$.
\medskip

For the first integral, expanding $P_\nu h$ in Taylor series around $(t,y,z)$, as well as expanding $G_J$ as $\rho\to 0$, 
we find no contribution to terms $x^{\pm \nu +1/2}$ and $\sqrt{x}\log(x)$ in the asymptotic expansion as $x\to 0$, unless $\nu=1/2$. 
In case $\nu=1/2$, we have $(-\nu+3/2)=\nu+1/2$ 
and hence the first integral contributes $\theta''_1 \cdot P_\nu h$
to the coefficient $c^+(F_\nu(h))$, with no contribution to $c^-(F_\nu(h))$.
In case $\nu\neq 1/2$, the first integral does not contribute to the coefficients $c^\pm(F_\nu(h))$
and we set $\theta''_1=0$.
\medskip

For the latter two integrals the discussion is parallel to that of $F''_0$ and we obtain
(we again abuse the notation by incorporating the 
$\rho-$factors into the kernels $G_*$)
\begin{align*}
c^-(F''_1)=0, \quad c^+(F''_1)&= G_{L,a} * P_\nu h + G_{L,b} * \partial_y P_\nu h + \textup{const} \cdot P_\nu h \\
&= \int_0^t \int_B  G_{L,a}(\wt, y,\wy) \, P_\nu h (t-\wt, \wy, z) \, d\wt \, \textup{dvol}_B(\wy) \\
&+ \int_0^t \int_B  G_{L,b}(\wt, y,\wy) \, \partial_y P_\nu h (t-\wt, \wy, z) \, d\wt \, \textup{dvol}_B(\wy) \\
& + \theta''_1 \cdot P_\nu h (t, y, z)
\end{align*}
where $G_{L,a}, G_{L,b}$ lift to a polyhomogeneous function on the parabolic blowup of
$\R^+\times B^2$ around $Y:=\{(t, y, \wy) \in \R^+\times B^2 \mid  t=0, y=\wy\}$
of leading order $(-2\nu-b)$ at the front face. \medskip

\textbf{Contribution from $F''_2$.} As before we obtain by construction
\begin{equation*}
\begin{split}
F'_2 &= c_\nu \int_{x^2}^t \int_{B}\kappa_2(\wt, x,y,\wy) \, P_\nu h(t-\wt, \wy, z) \, 
d\wt \, d\textup{vol}_{B}(\wy) \\ & = x^{\nu+1/2} \int_x^{\sqrt{t}} \int_{\R^b} 
\rho^{-2\nu+1} G''_2(\rho, \xi, u, y) P_\nu h(t-\rho^2, y-\rho u, z)\, v(y-\rho u)\, d\rho \, du 
\end{split}
\end{equation*}
where the kernel $G''_2$ is bounded and polyhomogeneous in $\rho$. 
Its discussion is parallel to that of $F''_0$ and we obtain
(we again abuse the notation by incorporating the 
$\rho-$factors into the kernel $G''_2$)
\begin{align*}
c^-(F''_2)=0, \quad c^+(F''_2)&= G''_2 * P_\nu h = 
\int_0^t \int_B  G''_2 (\wt, y,\wy) \, P_\nu h (t-\wt, \wy, z) \, d\wt \, \textup{dvol}_B(\wy).
\end{align*}
where $G''_2$ lifts to a polyhomogeneous function on the parabolic blowup of
$\R^+\times B^2$ around $Y:=\{(t, y, \wy) \in \R^+\times B^2 \mid  t=0, y=\wy\}$
of leading order $(-2\nu-b)$ at the front face. This proves the statement with $G'_\nu=G_{L,a} + G''_2$, 
$G''_\nu=G''_0 + G_{L,b}$ and $\theta_\nu=\theta'_0 + \theta'_1 + \theta''_1$.
\end{proof}

%%%%%%%%%%%%%%%%%%
\section{Heat kernel for algebraic boundary conditions}\label{heat-section}
%%%%%%%%%%%%%%%%%%

In this section we finally employ the signaling solution to construct the heat kernel 
for the Hodge Laplacian $\Delta_\Gamma$ with algebraic boundary conditions $\Gamma$. Consider 
the increasing sequence of eigenvalues $\nu_j^2\in [0,1)$ with $j=1,..., q,$ of the tangential 
operator $A$, counted with their multiplicities. 
Consider $\phi\in C^\infty_0(M)$ and put $u=\HF \phi$.
We seek to correct $u$ to satisfy algebraic boundary conditions $\Gamma$, i.e.
\begin{equation}\label{bcg1}
w = u + \sum_{j=1}^q F_{\nu_j}(h_{\nu_j}) \in \mathscr{D}_\Gamma (\Delta).
\end{equation}
where each $h_{\nu_j}$ lies in $\textup{Im}P_{\nu_j}$. 
Recall $\Gamma = (\Gamma_{ij}) \in \textup{Matr}(q,\Lambda_q)$ with 
diagonal entries given by $\Gamma_{jj}=b_{jj} \psi_j^{-} + \theta_{jj} \psi_j^{+}$, 
and the off-diagonal entries $\Gamma_{ij}= \theta_{ij} \psi_j^{+}$. The coefficients  
$b_{ij}, \theta_{ij}\in \R$ are such that either $b_{ii}=1$, or $b_{ii}=0$, where in the
latter case we require $\theta_{ii}=1$ and $\theta_{ij}=0$ for $i\neq j$. 
Here we assume that $b_{jj}=1$ for every $j=1,..,q$,
since in case of $b_{jj}=0$, $h_{\nu_j}=c^+_{\nu_j}(u)$. Then 
\eqref{bcg1} reads as follows
\begin{equation}\label{bcg}
c^+_{\nu_i}(u)+ c^+(F_{\nu_i}(h_{\nu_i})) = 
\sum_{j=1}^q \theta_{ij} h_{\nu_j}, \quad i=1, ... , q.
\end{equation}
We define $q\times q$ matrix valued operators (acting by convolution)
\begin{align*}
G^N &:= (\theta_{ij})_{ij} - \textup{diag}(G^N_{\nu_1}, ... , G^N_{\nu_q}) 
- \textup{diag}(\theta_{\nu_1}, ... , \theta_{\nu_q}). \\
G &:=\textup{diag} (  G'_{\nu_1} + G''_{\nu_1} * \partial_y, ... ,  G'_{\nu_q} + G''_{\nu_q} * \partial_y ).
\end{align*}
We also write 
\begin{align*}
H_\nu \phi &:= \textup{diag}(H_{\nu_1}\phi, .. ,H_{\nu_q}\phi), \\
h_\nu &:= \textup{diag}(h_{\nu_1}, .. ,h_{\nu_q}).
\end{align*}
Then we may rewrite \eqref{bcg} as $H_\nu \phi = (G^N - G) h_\nu$, 
where we note for $u=\HF\phi$ that $c^+_{\nu_i}(u)=H_{\nu_i}\phi$.
We apply on both sides of the relation the Fourier transform in $y\in \R^b$ 
and the Laplace transform in $t\in \R^+$, and obtain
\begin{equation}
\label{eqn}
\begin{split}
\mathscr{L} \circ \mathscr{F} H_\nu \phi &= \mathscr{L} \circ \mathscr{F} 
G^N \cdot \mathscr{L} \circ \mathscr{F} h_\nu 
- \mathscr{L} \circ \mathscr{F} (G h_\nu) \\
&= \mathscr{L} \circ \mathscr{F} G^N \cdot \left( \mathscr{L} \circ \mathscr{F} h_\nu - 
(\mathscr{L} \circ \mathscr{F} G^N)^{-1} \cdot \mathscr{L} \circ \mathscr{F} (G h_\nu)\right)
\\&= \widetilde{G}_N \cdot \left( \mathscr{L} \circ \mathscr{F} h_\nu - 
\widetilde{G}_N^{-1} \cdot \mathscr{L} \circ \mathscr{F} (G h_\nu)\right)
\end{split}
\end{equation}
where we have set
\begin{align*}
\widetilde{G}_N := \mathscr{L} \circ \mathscr{F} G^N = (\theta_{ij})_{ij} - 
\textup{diag}(\mathscr{L} \circ \mathscr{F}G^N_{\nu_1}, ... , \mathscr{L} \circ \mathscr{F}G^N_{\nu_q}) 
- \textup{diag}(\theta_{\nu_1}, ... , \theta_{\nu_q}).
\end{align*}
Recall from Proposition \ref{model-F} that $\widetilde{G}_N$ is a function of $(\zeta + |\w|^2)$. 
Existence and the structure of the inverse matrix $\widetilde{G}_N(\zeta + |\w|^2)^{-1}$
is the core of the subsequent section. Put 
\begin{align*}
D\equiv (D_{ij})_{ij}:=\mathscr{F}^{-1}\circ \mathscr{L}^{-1} \, \widetilde{G}_N^{-1} \, \mathscr{L} \circ \mathscr{F}.
\end{align*}
Then, applying the inverse $\widetilde{G}_N(\zeta + |\w|^2)^{-1}$ on both sides of \eqref{eqn}, we arrive at the following
relation
\begin{align}\label{eqn1}
D H_\nu \phi =  h_\nu - D \circ G h_\nu.
\end{align}
In the setup of isolated conical singularities, the operator $G$ is absent and 
\eqref{eqn1} provides an explicit result for $h_\nu$, obtained in one step by inverting the corresponding matrix. 
Unlike in the case of isolated conical singularities, here $G$ is a non trivial operator on the base manifold
$B$ and the solution is obtained from \eqref{eqn1} by an iterative procedure. We define for any $M\in \N$
\begin{equation}\label{LL}
\begin{split}
h^M_\nu \equiv \left(\begin{array}{c}h^M_{\nu_1}\\ : \\ h^M_{\nu_q} \end{array}\right) &:=
\left[\sum_{k=0}^M (D\circ G)^k \circ D \right] H_\nu \phi \\
&= \left(\sum_{j=1}^q (K^M_\Gamma)_{1j} * H_{\nu_j}\phi, \ \cdots \ , \sum_{j=1}^q 
(K^M_\Gamma)_{qj} * H_{\nu_j}\phi \right)^t,
\end{split}
\end{equation}
where we have introduced $K^M_\Gamma\equiv ((K^M_\Gamma)_{ij})_{ij} := \sum_{k=0}^M (D \circ G)^k \circ D$.
This defines an approximate solution to \eqref{eqn1}, solving it up to an error
\begin{align}\label{eqn2}
D H_\nu \phi =  h^M_\nu - D \circ G h^M_\nu - (D\circ G)^{M+1} D H_\nu \phi.
\end{align}
Below we show that the error term vanishes and $h^M_\nu$ converges
as $M\to \infty$. Anticipating that discussion, we may define $K^\infty_\Gamma\equiv ((K^\infty_\Gamma)_{ij})_{ij} := \sum_{k=0}^\infty (D \circ G)^k \circ D$, and the heat kernel for algebraic boundary conditions
$\Gamma$ is subsequently given by 
\begin{align}
\HG=\HF + \sum_{i,j=1}^q H_{\nu_i} * (K^\infty_\Gamma)_{ij} *H_{\nu_j},
\end{align}
where the kernels are convolved in $t$ and concatenated over $B$.
\medskip

The remainder of the section is concerned with the analysis of the operator
orders in the definition of $h^M$. The integral kernel of $D$ is explicitly given by
($u=(y-\wy)/\sqrt{t}$)
\begin{equation}\label{D}
\begin{split}
D(t,y,\wy) &= -i(2\pi)^{-b-1}\int_{\R^b} \int_{i\R+\delta} e^{-i\w (y-\wy)} e^{t\zeta} 
\widetilde{G}_N^{-1}(\zeta + |\w|^2) d\zeta \, d\w \\
&=  -i(2\pi)^{-b-1}\int_{\R^b} \int_{i\R+\delta} e^{-i\w (y-\wy)} e^{-t|\w|^2} 
e^{st}\widetilde{G}_N^{-1}(s) ds \, d\w \\
&=  -i(2\pi)^{-b-1}(\sqrt{t})^{-b} \mathscr{L}^{-1} \widetilde{G}_N^{-1}(t) \cdot \int_{\R^b} e^{-i w \cdot u} e^{- |w|^2} dw.
\end{split}
\end{equation}
The asymptotics of $ \mathscr{L}^{-1} \widetilde{G}_N^{-1}(t)$ as $t\to 0$ is established
in Theorem \ref{K} below. More precisely, noting the $p\times p$ matrix structure of the 
operators, we obtain for any $(ij) \in \{1,..,q\}^2$
$$
\mathscr{L}^{-1} \widetilde{G}_N^{-1}(t)_{ij} \sim (\sqrt{t})^{-2+2\nu_{ij}}, \quad  t\to 0,
$$
where $\nu_{ij}=\nu_i + \nu_j$ if $i\neq j$ and $\nu_{ij}=\nu_i$ if $i=j$.
We neglect eventual logarithmic terms in the exact $t$-asymptotics of 
$\mathscr{L}^{-1} \widetilde{G}_N^{-1}(t)$ up to the end of this section, since for the argument 
on convergence of $h^M$ as $M\to \infty$, only the leading terms are relevant. Keeping these 
logarithmic terms out of the picture, the component $D_{ij}$ of the matrix-valued kernel $D$ lifts to a
function on the parabolic blowup of
$\R^+\times B^2$ around $Y:=\{(t, y, \wy) \in \R^+\times B^2 \mid  t=0, y=\wy\}$
of leading order $(-2-b+2\nu_{ij})$ at the front face. The projective coordinates on the blowup are given in \eqref{left-coord}. 

\begin{prop}\label{G-composition} 
Let $G_1$ and $G_2$ lift to polyhomogeneous functions on the parabolic blowup of
$\R^+\times B^2$ around $Y:=\{(t, y, \wy) \in \R^+\times B^2 \mid  t=0, y=\wy\}$
of leading orders $(-2-b+\A_1)$ and $(-2-b+\A_2)$ at the front face, respectively. Then
$$
G_1 *_t G_2 (t, y, y') = \int_0^t \int_B G_1(t-\wt, y, \wy) \, G_2(\wt, \wy, y') \, d\wt\, d\textup{vol}_B(\wy),
$$
lifts to a polyhomogeneous function on $[\R^+_{\sqrt{t}}\times B^2,Y]$
of leading order $(-2-b+\A_1+\A_2)$.
\end{prop}

\begin{proof}
By polyhomogeneity of $G_{1,2}$ we may write for $i=1,2$ and any $N\in \N, c>0$
\begin{equation*}
G_i (t, y, \wy) = \sum_{k=0}^{N-1} G^k_i(t,y,\wy) + \widetilde{G}^N_i (t, y, \wy), \
\left\{ \begin{split}
&G^k_i(c^2 t, cy, c\wy) = c^{-2-b+\A_1+k} G^k_i (t, y, \wy), \\
&\widetilde{G}^N_i (c^2 t, cy, c\wy) = O(c^{-2-b+\A_1+N}), \ c \to 0.
\end{split} \right.
\end{equation*}
For a composition of individual homogeneous summands we obtain
\begin{align*}
G^k_1 *_t G^l_2 (c^2t, cy, cy') &= \int_0^{c^2t} 
\int_B G^k_1(c^2t-\wt, cy, \wy) \, G^l_2(\wt, \wy, cy') \, d\wt\, d\textup{vol}_B(\wy) \\
&= c^{2+b} \int_0^{t} \int_B G^k_1(c^2t-c^2\wt, cy, c\wy) \, G^l_2(c^2\wt, c\wy, cy') \, d\wt\, d\textup{vol}_B(\wy)\\
&= c^{-2-b+\A_1+\A_2+k+l} G^k_1 *_t G^l_2 (t, y, y').
\end{align*}
Similar estimates for the remainder terms $\widetilde{G}^N_{1,2}$ prove the statement.
\end{proof}

The statement on the leading orders in Proposition \ref{G-composition} 
above holds also without the assumption of polyhomogeneity. Moreover, 
presence of $\partial_y$ derivatives is not excluded from that picture, 
since $\partial_y$ corresponds to lowering the leading order at the front 
face of $[\R^+\times B^2, Y]$ by one. Consequently, $G$ is zero off diagonal and its 
diagonal components $G_{jj} = (G'_{\nu_j}+ G''_{\nu_j} * \partial_y)$ may be viewed as of leading 
order $(-1-b+2\nu_j)$ at the front face, when lifted to $[\R^+ \times B^2, Y]$. 
\medskip

Hence, by Proposition \ref{G-composition}, we find that $(D\circ G)_{ij} = D_{ij} \circ G_{jj}$
lifts to a polyhomogeneous distribution on $[\R^+ \times B^2, Y]$ of leading order 
$(-2-b+\A_{ij})$, where $\A_{ij}:=1+2(\nu_{ij}-\nu_j) \geq 1$. Due to $\A_{ij} \geq 1$, 
subsequent $M$-fold compositions improve the order and hence $h^M_\nu$ converges and
the error in \eqref{eqn2} vanishes as $M \to \infty$.

%%%%%%%%%%%%%%%%%%
\section{Heat-trace asymptotic expansions}\label{trace-section}
%%%%%%%%%%%%%%%%%%

In view of \eqref{D}, we begin with studying $\widetilde{G}_N^{-1}$ and its inverse Laplace transform.
Fix the branch of logarithm in $\C \backslash \R^-$ for the definition of powers 
of complex numbers. We apply the inversion rule using adjuncts
\begin{align}
(\widetilde{G}_N^{-1})_{ij} = (-1)^{i+j}\frac{(\textup{adj}\widetilde{G}_N)_{ij}(\zeta)}{\det \widetilde{G}_N(\zeta)},
\end{align}
where $(\textup{adj}\widetilde{G}_N)_{ij}(\zeta)$ is the determinant of the reduced matrix, obtained 
from $\widetilde{G}_N$ by deleting the $i$-th row and $j$-th column. 
Let $\{\nu_j\}_{j=1}^q$ be ordered in the ascending order and $p\leq q$ denote the multiplicity 
of zero in $\textup{Spec}(A)$, so that $\nu_1=..=\nu_p=0$. Write $\theta_j:= \theta_{\nu_j}-\theta_{jj}, j=1,..,p$.
Set $\kappa_\theta := 2 (\gamma - \log 2 + \theta)$ 
and introduce the multi-index notation, setting for $\A=(\A_1,...,\A_p) \in \Z^p$ and $\beta = (\beta_{p+1}, ..,\beta_q)\in \Z^{q-p}$
$$
(\log\zeta + \kappa_{\theta})^{\A}= 
\prod\limits_{k=1}^p (\log\zeta + \kappa_{\theta_k})^{\A_k}, \quad 
\zeta_{\nu}^\beta = \prod\limits_{k=p+1}^q \zeta^{\beta_k \nu_k}\footnote{We set $\zeta^\beta_\nu=1$ in case $p=q$.}.
$$
Let tuples with each entry given by $1$, be denoted by $\one$. 
Then in view of the explicit formulas in Proposition \ref{model-F} we obtain
\begin{equation*}
\det \widetilde{G}_N(\zeta) = C \left((\log \zeta + \kappa_{\theta}) \, \zeta_\nu \right)^\one
\left(1+ \sum_{| \A|=0}^p \sum_{|\beta|=0}^{q-p} C_{\A \beta} (\log\zeta + \kappa_{\theta})^{-\A} \zeta_\nu^{-\beta}\right),
\end{equation*}
for certain coefficients $C, C_{\A \beta}$. The summation above excludes $|\A|=|\beta|=0$.
For $|\zeta| \gg 0$ we may expand $(\det \widetilde{G}_N(\zeta))^{-1}$ in Neumann series
and obtain 
\begin{equation*}
(\det \widetilde{G}_N(\zeta))^{-1} = C^{-1}\left( (\log \zeta + \kappa_{\theta}) \, \zeta_\nu \right)^{-\one} 
\left(1+ \sum_{|\A| + |\beta|=1}^\infty D_{\A \beta} (\log\zeta + \kappa_{\theta})^{-\A} \zeta_\nu^{-\beta}\right),
\end{equation*}
which is convergent for $|\zeta|\gg 0$ and is in particular an asymptotic series as $|\zeta| \to \infty$. Similar computations hold 
for $(\textup{adj}F)_{ij}(\zeta)$. Hence overall we obtain, cf.
(\cite{Moo:THK}, Lemma 8.6)

\begin{prop}\label{LK} There exist constants $A_{ij}$ and $B_{\A \beta}$ such that for $|\zeta| \gg 0$
\begin{align*}
(\widetilde{G}_N^{-1})_{ij}(\zeta) &= 
A_{ij}\left(1+ \sum_{|\A|+|\beta|=1}^\infty B_{\A \beta} (\log\zeta + \kappa_{\theta})^{-\A} \zeta^{-\beta}_\nu\right) \\
&\times  \left\{
\begin{array}{ll}
(\log \zeta +\kappa_i)^{-1} (\log \zeta +\kappa_j)^{-1}, \ &\textup{if} \ \nu_i=\nu_j=0, i\neq j, \\
(\log \zeta +\kappa_i)^{-1} \zeta^{-\nu_j}, \ &\textup{if} \ \nu_i=0, \nu_j\neq 0, \\
(\log \zeta +\kappa_i)^{-1}, \ &\textup{if} \ \nu_i=\nu_j=0, i=j, \\
\zeta^{-\nu_i-\nu_j}, \ &\textup{if} \ \nu_i, \nu_j \neq 0, i\neq j, \\
\zeta^{-\nu_i}, \ &\textup{if} \ \nu_i=\nu_j\neq 0, i=j.
\end{array}\right.
\end{align*}
\end{prop}

We point out that no $(\log \zeta)$ factors arise if $\Delta_\Gamma$ coincides with the 
Friedrichs extension of the Hodge Laplacian when restricted to the zero eigenspace of $A$.
This case exhibits only a classical expansion without 
the \emph{exotic} phenomena and we do not consider it here. We now want to 
derive an asymptotic expansion for $\mathscr{L}^{-1} (\mathscr{T}_{\A,\beta}) (t)$ with
\begin{align*}
\mathscr{T}_{\A,\beta} (\zeta)
:= (\log\zeta + \kappa_{\theta})^{-\A} \zeta_\nu^{-\beta}.
\end{align*}

\begin{thm}\label{K}
Write $\nu_\beta:=\beta_{p+1}\nu_{p+1} + .. + \beta_q\nu_q$ for the given $\beta\in \N^{q-p}$.
We set $\nu_\beta=0$ in case $p=q$.
For $\A\in \N^p$ and $\beta\in \N^{q-p}$ we then have up to smooth additive components
\begin{equation*}
\mathscr{L}^{-1} (\mathscr{T}_{\A,\beta}) (t) \sim \sum_{k=0}^\infty 
E_{\beta k}\, t^{-1+\nu_\beta} \log^{-|\A|-k}(t), \ t\to 0.
\end{equation*}
\end{thm}

\begin{proof}
As the first step we deform the integration region $(i\R+\delta)$ to 
$\mu$, concatenated out of three parts, $\mu_1=i(-\infty, 1]$, 
the half circle $\mu_2=\{\zeta \in \C \mid |\zeta|=1, \textup{Re} (\zeta) \geq 0\}$ oriented
counterclockwise, and $\mu_3=i[1,\infty)$. The change of the 
integration contour is possible, since for some constant $C>0$
\begin{align*}
\left| \, \frac{1}{2\pi i} \int_0^\delta e^{t(x \pm iR)} \mathscr{T}_{\A, \beta} (i R+x) \, dx \,\right| 
\leq C(\log R)^{-|\A|} R^{-\nu_\beta} \to 0, \ \textup{as} \ R\to \infty.
\end{align*}
Integration over $\mu_2$ leads to a function 
$\mathscr{L}^{-1} (\mathscr{T}_{\A, \beta})_{\mu_2} \in C^\infty[0,\infty)$, 
so that we may write
\begin{align*}
\mathscr{L}^{-1} (\mathscr{T}_{\A, \beta}) (t) &=\mathscr{L}^{-1} (\mathscr{T}_{\A, \beta})_{\mu_2}  + 
\frac{1}{2\pi i} \int_ {\mu_1 \cup \mu_3} e^{t\zeta}  \mathscr{T}_{\A, \beta} (\zeta) \, d\zeta \\
&= \mathscr{L}^{-1} (\mathscr{T}_{\A, \beta})_{\mu_2}  + 
\frac{1}{2\pi} \int_1^\infty e^{i t x}  \mathscr{T}_{\A, \beta} (ix) \, dx 
- \frac{1}{2\pi} \int_1^\infty e^{-i t x}  \mathscr{T}_{\A, \beta} (-ix) \, dx .
\end{align*}
For the second summand we rotate the integration contour to $i[1,\infty)$. 
For the third summand we rotate the integration contour to $(-i[1,\infty))$. 
We make the argument explicit for the second summand.
Let  $R>1$ and the contour $\eta_R := \{R \exp(i\phi) \mid \phi \in [0,\pi/2]\}$ be oriented counterclockwise. 

\begin{figure}[h]
\begin{center}
\begin{tikzpicture}
\draw (-2,0) -- (2,0);
\draw (0,-1) -- (0,2);

\draw (0.5,0) -- (2,0);
\draw (1.5,0) .. controls (1.5,0.8) and (0.8,1.5) .. (0,1.5);

\draw (0,1.5) -- (0.3,1.55);
\draw (0,1.5) -- (0.25,1.4);

\node at (1.5,-0.5) {$R$};
\node at (1.5,1.3) {$\eta_R$};

\end{tikzpicture}
\end{center}
\label{contour}
\caption{The integration contour $\eta_R$.}
\end{figure}
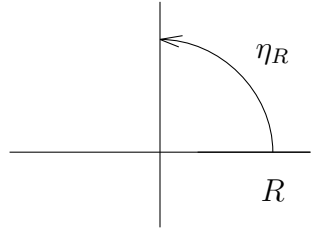

We use the $O$-notation for the asymptotics as $R\to \infty$.
Then, substituting $x=R \exp(i\phi)$, we find for some $C,C'>0$
\begin{align*}
\mid\int_{\eta_R}  e^{itx} \mathscr{T}_{\A, \beta}(ix) \, dx \mid 
 &= \, \mid R \int_0^{\pi/2}  \exp{(-t R\sin\phi + i(\phi + t R\cos\phi))}
\mathscr{T}_{\A, \beta} (iR e^{i\phi}) \, d\phi   \mid \\
& \leq C\, R \, (\log R)^{-|\A|} R^{-\nu_\beta} \int_0^{\pi/2} \exp(-t R\sin\phi) \, d\phi \\
& = C\, R \, (\log R)^{-|\A|} R^{-\nu_\beta} \int_0^{\pi/4} \exp(-t R\sin\phi) \, d\phi + O(R^{-\infty}) 
\end{align*}
\begin{align*}
& \leq C'\, R \, (\log R)^{-|\A|} R^{-\nu_\beta} \int_0^{\pi/4} \exp(-t R\sin\phi) \, \cos(\phi)\, d\phi +
 O(R^{-\infty}) \\ & = C'\, R \, (\log R)^{-|\A|} R^{-\nu_\beta} \int_0^{\sin \pi /4} \exp(-t R y) \, dy  
+ O(R^{-\infty}) \\ &= O\left((\log R)^{-|\A|} R^{-\nu_\beta}\right)  \to 0, \ \textup{as} \ R\to \infty,
\end{align*}
where in the fourth line we have used the fact that $\cos \phi$ is bounded 
from below for $\phi\in [0,\pi/4]$, and in the fifth line we substituted $y=\sin \phi$.
Subsequently, writing $\overline{\eta_1}=\{\overline{z} \mid z\in \eta_1\}$ for the 
clockwise oriented contour, we find
\begin{align*}
\mathscr{L}^{-1} (\mathscr{T}_{\A, \beta}) (t) &=\mathscr{L}^{-1} (\mathscr{T}_{\A, \beta})_{\mu_2}
+ \frac{1}{2\pi} \int_{\eta_1} e^{i t x}  \mathscr{T}_{\A, \beta} (ix) \, dx
 - \frac{1}{2\pi} \int_{\overline{\eta_1}} e^{-i t x}  \mathscr{T}_{\A, \beta} (-ix) \, dx \\
&+ \frac{i}{2\pi} \int_1^\infty e^{-ty} \mathscr{T}_{\A, \beta} (-y)_+ \, dy
- \frac{i}{2\pi} \int_1^\infty e^{-ty} \mathscr{T}_{\A, \beta} (-y)_- \, dy,
\end{align*}
where the subindex $\pm$ indicates $\log(-y)=\log y \pm i\pi$, respectively, in the definition of 
$\mathscr{T}_{\A, \beta} (-y)$, and the first three summands define smooth functions $C^\infty[0,\infty)$.
\medskip

We continue under equivalence up to smooth functions. Then 
without writing out the smooth summands we obtain
\begin{align*}
\mathscr{L}^{-1} (\mathscr{T}_{\A, \beta}) (t) &= \frac{i}{2\pi}
 e^{-i\pi \nu_\beta} \int_1^\infty e^{-ty}  y^{-\nu_\beta}  (\log y + i\pi + \kappa_{\theta})^{-\A} dy \\
& - \frac{i}{2\pi}   e^{i\pi \nu_\beta} \int_1^\infty e^{-ty}  y^{-\nu_\beta}  (\log y - i\pi + \kappa_{\theta})^{-\A} dy
\\ &=: \frac{i}{2\pi} ( e^{-i\pi \nu_\beta}\mathscr{L}^+_{\beta}(t) -  e^{i\pi \nu_\beta}\mathscr{L}^-_{\beta}(t)).
\end{align*}

We establish an asymptotic expansion for each $\mathscr{L}^\pm_{\beta }(t)$ as $t\to 0$. 
For this we differentiate $[\nu_\beta]$ times\footnote{For any $r\geq 0$, [r] denotes the largest integer $\leq r$.} 
in $t$ and obtain by a change of variables $x=ty$
\begin{align*}
\frac{d^{[\nu_\beta]}}{dt^{[\nu_\beta]}}\mathscr{L}^\pm_{\beta}(t) &
=   (-1)^{[\nu_\beta]}\, \int_1^\infty e^{-ty}  y^{-\nu_\beta+ [\nu_\beta]}  (\log y \pm i\pi + \kappa_{\theta})^{-\A} dy \\ &=  (-1)^{[\nu_\beta]}\, \int_0^\infty e^{-ty}  y^{-\nu_\beta+ [\nu_\beta]}  (\log y \pm i\pi + \kappa_{\theta})^{-\A} dy + \textup{smooth}\\
&=  (-1)^{[\nu_\beta]}\, \int_0^\infty e^{-x}  x^{-\nu_\beta+ [\nu_\beta]}  
\left(\frac{\log (x)\pm i\pi + \kappa_{\theta}}{\log (t)} -1\right)^{-\A} dx \\ & \times t^{-1+\nu_\beta -[\nu_\beta] } \log^{-|\A|}(t) + \textup{smooth}
\end{align*}
Since $(1-r)^{-1}=\sum_{k=0}^M r^k + r^{M+1}(1-r)^{-1}$ for any $M\in \N$ and $r\neq 1$, we find, up to smooth additive components
\begin{align*}
\frac{d^{[\nu_\beta]}}{dt^{[\nu_\beta]}}\mathscr{L}^\pm_{\beta}(t) &\sim 
t^{-1+\nu_\beta -[\nu_\beta] } \log^{-|\A|}(t) \sum_{k=0}^\infty   D^\pm_{\beta k} \log^{-k}(t), \ t\to 0.
\end{align*}
In order to integrate the asymptotic series $[\nu_\beta]$ times, note for any $M>-1 $ and $N\in \N$
\begin{align*}
\int_0^t \tau^M \log^{-N}(\tau) \, d\tau &= \int_0^t (M+1)^{-1}(\tau^{M+1}\log^{-N}(\tau))' \, d\tau + 
\int_0^t \frac{N}{M+1} \tau^M \log^{-N-1}(\tau)\, d\tau \\ &= (M+1)^{-1}t^{M+1}\log^{-N}(t) 
+ O(t^{M+1}\log^{-N-1}(t)), \ t\to 0.
\end{align*}
In case $\nu_\beta\in \N$, we also need to consider the case of $M=-1$ and note for $t_0>t$
\begin{align*}
\int_t^{t_0} \tau^{-1} \log^{-N}(\tau) \, d\tau = \frac{\log^{-N-1}(t) - \log^{-N-1}(t_0)}{N+1}.
\end{align*}
Iterating this argument we finally arrive, up to smooth additive components, at the asymptotic expansion
\begin{align*}
\mathscr{L}^\pm_{\beta}(t) &\sim 
t^{-1+\nu_\beta} \log^{-|\A|}(t) \sum_{k=0}^\infty   E^\pm_{\beta k} \log^{-k}(t), \ t\to 0.
\end{align*}
This proves the statement.
\end{proof}

Let $K$ lift to polyhomogeneous function on the parabolic blowup of
$\R^+\times B^2$ around $Y:=\{(t, y, \wy) \in \R^+\times B^2 \mid  t=0, y=\wy\}$
of leading order $(-2-b+\gamma)$ at the front face. By an ad verbatim extension of Proposition \ref{G-composition}, 
compare also (\cite{Moo:THK}, Proposition 8.7), we find that the convolution
$ H_{\nu_i} * K * H_{\nu_j} $ lifts to a polyhomogeneous distribution
on the parabolic blowup space $\mathscr{M}^2_h$, 
of leading order $(-1-b-(\nu_i+\nu_j) + \gamma)$ at the front face,
which vanishes to infinite order at $\tf$ and $\td$ and is of leading order $(\nu_i+1/2), (\nu_j+1/2)$ at $\lf, \rf$, 
respectively. By standard pushforward arguments, cf. \cite[Section 4]{MazVer:ATM}, we obtain
\begin{align}\label{trace}
\textup{Tr} \, H_{\nu_i} * K * H_{\nu_j} (t) 
\sim \sqrt{t}^{\ \gamma - (\nu_i+\nu_j) } 
\sum_{k=0}^\infty d_k \, \sqrt{t}^{\ k}, \ \textup{as} \ t\to 0.
\end{align}
Combining Theorem \ref{K} with \eqref{trace}, we obtain the following

\begin{prop}\label{GHT}
Let $K$ lift to polyhomogeneous function on the parabolic blowup of
$\R^+\times B^2$ around $Y:=\{(t, y, \wy) \in \R^+\times B^2 \mid  t=0, y=\wy\}$
of leading order $(-2-b+\gamma)$. Write $\nu_\beta:=\beta_{p+1}\nu_{p+1} + .. + \beta_q\nu_q$ for any given $\beta\in \N^{q-p}$. For $\A\in \N^p$ and $\beta\in \N^{q-p}$ we put $\mathscr{T}_{\A,\beta}=
(\log\zeta + \kappa_{\theta})^{-\A} \zeta_\nu^{-\beta}$. Then as $t\to 0$ we obtain
\begin{align*}
\textup{Tr} \, \Big(H_{\nu_i} * \mathscr{L}^{-1} (\mathscr{T}_{\A,\beta}) * K * H_{\nu_j}\Big) (t) 
\sim \sqrt{t}^{\ \gamma - (\nu_i+\nu_j) + 2\nu_\beta} 
\sum_{k,l=0}^\infty d'_k \, \sqrt{t}^{\ k} \log^{-|\A|-l}(t).
\end{align*}
\end{prop}

\begin{proof}
The statement follows by an iterative application of the following identity
\begin{align*}
&\int_0^t \tau^\sigma \log^{-\rho}(\tau) (t-\tau)^\mu \, d\tau = 
\sum_{k=0}^\infty (-1)^k \, t^{\mu-k} \, \binom{\mu}{k} \int_0^t \tau^{\sigma + k}  \log^{-\rho}(\tau)  \, d\tau \\
&= \sum_{k=0}^\infty \frac{(-1)^k}{\sigma + k + 1} \binom{\mu}{k} \Big( t^{\sigma + \mu +1} \,  \log^{-\rho}(t) 
+ \rho \ t^{\mu-k} \, \int_0^t \tau^{\sigma + k}  \log^{-\rho-1}(\tau)  \, d\tau \Big)\\
&= \sum_{k=0}^\infty \frac{(-1)^k}{\sigma + k + 1} \binom{\mu}{k} \Big( t^{\sigma + \mu +1} \,  \log^{-\rho}(t) 
+O(t^{\sigma + \mu +1} \,  \log^{-\rho-1}(t)) \Big), \ t\to 0.
\end{align*}
\end{proof}

Combining Proposition \ref{LK} with Proposition \ref{GHT} 
we obtain a full asymptotic expansion for each $\textup{Tr} H_{\nu_i} *_t  (K^\infty_\Gamma)_{ij} *_t H_{\nu_j}$, 
complicated only by the intricate combination of the various components after matrix multiplications. The heat trace
asymptotic expansion admits logarithmic terms in accordance to \cite{KLP:EEA} and \cite{GKM:TEF}.
We make the leading orders explicit in our final main result.

\begin{thm}\label{main}
\begin{align*}
\textup{Tr} H_{\nu_i} *_t  (K^\infty_\Gamma)_{ij} *_t H_{\nu_j} \sim_{t\to 0}  \left\{
\begin{array}{ll}
O(\log^{-2}(t)), \ &\textup{if} \ \nu_i=\nu_j=0, i\neq j, \\
O(\sqrt{t}^{\ \nu_j}\log^{-1}(t)), \ &\textup{if} \ \nu_i=0, \nu_j\neq 0, \\
O(\log^{-1}(t)), \ &\textup{if} \ \nu_i=\nu_j=0, i=j, \\
O(\sqrt{t}^{\ \nu_i + \nu_j}), \ &\textup{if} \ \nu_i, \nu_j \neq 0, i\neq j, \\
O(\sqrt{t}^{\ 0}), \ &\textup{if} \ \nu_i=\nu_j\neq 0, i=j.
\end{array}\right.
\end{align*}
\end{thm}

This corresponds to (\cite{Moo:THK},Theorem 8.2), up to the case $\nu_i=\nu_j\neq 0, i=j$. 
Note for example that for $\nu_i, \nu_j \neq 0$ and $i\neq j$ the correcting kernel lies according to (\cite{Moo:THK}, Theorem 8.2) in
the space $\dot{\Phi}^{\nu_i+\nu_j+2}$, with the front face leading order $-3+ \nu_i+\nu_j+2=\nu_i+\nu_j-1$,
which yields the leading order $(\nu_i+\nu_j)/2$ after taking traces. Here, Theorem \ref{main} and (\cite{Moo:THK}, Theorem 8.2)
agree. \medskip

However, in case $i=j$, asymptotics of $\widetilde{G}_N^{-1}(\zeta)$ in Proposition \ref{LK}
is different and hence Theorem \ref{main} yields the leading order $0$, different from the expansion for $i\neq j$,
in contrast to (\cite{Moo:THK}, Theorem 8.2) which erroneously asserts a heat trace expansion of order $(\nu_i+\nu_j)$. 

\section*{Acknowledgements}
The author would like to thank Rafe Mazzeo for his continuous support and mathematical insight, 
as well as Thomas Krainer for helpful remarks. He is indebted to the anonymous referee for the careful reading of the manuscript and several important improvements. The author gratefully 
acknowledges the support by the Hausdorff Research Institute at
the University of Bonn.

\def\cprime{$'$}
\providecommand{\bysame}{\leavevmode\hbox to3em{\hrulefill}\thinspace}
\providecommand{\MR}{\relax\ifhmode\unskip\space\fi MR }
\providecommand{\MRhref}[2]{%
  \href{http://www.ams.org/mathscinet-getitem?mr=#1}{#2}
}
\providecommand{\href}[2]{#2}

\listoffigures

\end{document}